\documentclass[11pt]{amsart}
\usepackage{amssymb}
\usepackage{amsmath}
\usepackage{amsfonts}
\usepackage{mathrsfs}
\usepackage{psfrag}
\usepackage{verbatim}
\usepackage{color}
\usepackage{enumitem}
\usepackage{bm}
\usepackage[usenames,dvipsnames]{xcolor}
\usepackage{dsfont}
\usepackage{overpic}
\usepackage{subcaption}
\usepackage{multirow}

\newtheorem{assump}{Assumption}

\usepackage[colorlinks=true]{hyperref}
\hypersetup{urlcolor=blue, citecolor=cyan}

\DeclareGraphicsExtensions{}
\theoremstyle{plain}
\newtheorem{theorem}{Theorem}[section]
\newtheorem{corollary}[theorem]{Corollary}

\newtheorem{lemma}[theorem]{Lemma}
\newtheorem{proposition}[theorem]{Proposition}

\theoremstyle{definition}

\theoremstyle{remark}
\newtheorem{remark}[theorem]{Remark}

\newcommand{\eps}{\varepsilon}

\newcommand{\BR}{\mathbb{R}}
\newcommand{\BP}{\mathbb{P}}
\newcommand{\BE}{\mathbb{E}}
\newcommand{\BZ}{\mathbb{Z}}
\newcommand{\BN}{\mathbb{N}}

\newcommand{\sfM}{\mathsf{M}}

\newcommand{\cc}{\mathsf{c}}

\newcommand{\ind}{\mathds{1}}

\begin{document}
\title[Stochastic hybrid systems with fast sampling effects]{Uniform-in-time bounds for a stochastic hybrid system with fast periodic sampling and small white-noise
}

\author[Shivam Singh Dhama and Konstantinos Spiliopoulos]{}
\address{Department of Mathematics and Statistics\\ Boston University, Boston, {\sc ma}, 02215, {\sc usa}}
\email{ssdhama@bu.edu, shivamsd.maths@gmail.com}

\address{Department of Mathematics and Statistics\\ Boston University, Boston, {\sc ma}, 02215, {\sc usa}}
 \email{kspiliop@math.bu.edu}
\thanks{$^*$Corresponding author: Shivam Singh Dhama}

\date{\today}

\maketitle

\centerline{\scshape Shivam Singh Dhama$^*$ and Konstantinos Spiliopoulos}
\medskip
{\footnotesize

 \centerline{Department of Mathematics and Statistics, Boston University, Boston, Massachusetts, 02215, {\sc usa}}
}

%


\begin{center}
\rule{17cm}{.01mm}
\end{center}
\begin{abstract}
We study the asymptotic behavior, uniform-in-time, of a non-linear dynamical system under the combined effects of fast periodic sampling with period $\delta$ and small white noise of size $\eps,\thinspace 0<\eps,\delta \ll 1$. The dynamics depend on both the current and recent measurements of the state, and as such it is not Markovian. Our main results can be interpreted as Law of Large Numbers ({\sc lln})  and Central Limit Theorem ({\sc clt}) type results. {\sc lln}  type result shows that the resulting stochastic process is close to an ordinary differential equation ({\sc ode}) uniformly in time as $\eps,\delta \searrow 0.$ Further, in regards to {\sc clt}, we provide quantitative and uniform-in-time control of the fluctuations process. The interaction of the small parameters provides an additional drift term in the limiting fluctuations, which captures both the sampling and noise effects. As a consequence, we obtain a first-order perturbation expansion of the stochastic process along with time-independent estimates on the remainder. The zeroth- and first-order terms in the expansion are given by an {\sc ode} and {\sc sde}, respectively. Simulation studies that illustrate and supplement the theoretical results are also provided.
\end{abstract}
\begin{center}
\rule{17cm}{.01mm}
\end{center}


\section{Introduction}
In this paper, we consider the asymptotic analysis of a controlled stochastic differential equation ({\sc sde})
\begin{equation}\label{E:Intro-SDE}
dX_t^{\eps,\delta}=\left\{f(X_t^{\eps,\delta})+\kappa(X_{\pi_\delta(t)}^{\eps,\delta})\right\}dt+\eps \sigma(X_t^{\eps,\delta})dW_t, \quad X_0^{\eps,\delta}=x_0,
\end{equation}
as $\eps,\delta \searrow 0$ and $t\in [0,\infty)$. Here, the parameters $\eps$ and $\delta=\delta(\eps),\thinspace 0< \eps,\delta \ll 1,$ correspond to the size of noise and the rate of state measurements, respectively. The operator $\pi_\delta(t) \triangleq \delta \lfloor {t}/{\delta} \rfloor$ transforms the continuous time $t\in [0,\infty)$ to the integer multiple of $\delta$ and the process $W=\{W_t: t \ge 0\}$ represents a standard Brownian motion.  For the detailed assumptions on the coefficients of (\ref{E:Intro-SDE}) see Section \ref{S:Problem-setup}. In our problem set-up, equation \eqref{E:Intro-SDE} can be thought of as a small random perturbation of a non-linear control system
\begin{equation}\label{E:Intro-det-sys}
{dx}=\{f(x)+u\}dt, \quad x(0)= x_0
\end{equation}
with a feedback control law $u=\kappa(x)$ using its sample-and-hold implementation. In our model, by the sample-and-hold implementation of the control function $u$, we mean $u$ is updated through state measurements at the time instants $k\delta$ and it remains fixed throughout in the sampling interval $[k \delta, (k+1)\delta), \thinspace k\in \BZ^+$. We note that the instantaneous value of $dX_t^{\eps,\delta}$ in \eqref{E:Intro-SDE} depends on both the current state value $X_t^{\eps,\delta}$ and the most recent measurement $X_{\pi_\delta(t)}^{\eps,\delta}.$ Hence, it is not a Markovian process.

The study of system \eqref{E:Intro-SDE} is important as it appears very frequently in control theory where a system of the form $dx=f(x)dt$ is controlled by a digital computer and is perturbed by a small white noise. In such cases, the state of the system is evolved continuously with time whereas the control function requires the state values (samples) only at certain discrete-time instances. This interaction of state evolution at discrete- and continuous-time makes the system a hybrid dynamical system ({\sc hds})\cite{GST-HybDynSys-book}. An example of {\sc hds} is sampled-data systems \cite{YuzGoodwin-book} where the measurements of the state are only available at discrete time instants and control input is updated using a digital computer. Much of the study on these systems focused on stability analysis (e.g., \cite{NesicTeelCarnevale-TAC2009,YuzGoodwin-book}). To the best of our knowledge, there is very little work on the interaction of sampling and noise (see, \cite{dhama2021asymptotic,dhama2023fluctuation} for the recent results in this direction). Hence, in this paper, we aim to explore the asymptotic analysis of a {\sc hds} under the combined effects of sampling and noise, uniform-in-time.



Part of the contribution of this work is in studying the asymptotic analysis, uniform-in-time ({\sc uit}), of the non-linear\footnote{Equation \eqref{E:Intro-SDE} is classified as a non-linear {\sc sde} because the functions $f$, $\kappa$ and $\sigma$ could be any non-linear functions satisfying Assumptions \ref{A:Poly-Lip}, \ref{A:Poly-growth} and \ref{A:Poly-Int}.} {\sc sde} \eqref{E:Intro-SDE} under the combined influence of fast ($\delta\searrow 0$) periodic sampling and small ($\eps \searrow 0$) state-dependent white-noise. 
As $\eps,\delta \searrow 0,$ we found that the resulting stochastic process $X_t^{\eps,\delta}$ and its rescaled fluctuation process $Z_t^{\eps,\delta}\triangleq \eps^{-1}(X_t^{\eps,\delta}-x_t)$ are close to an ordinary differential equation ({\sc ode}) and {\sc sde}, respectively. The {\sc ode} describing the mean behavior is the closed-loop {\sc ode}, regardless of how $\eps,\delta \searrow 0$; however, the fluctuation behavior is found to depend on the relative rates at which $\eps,\delta \searrow 0$. In the fluctuation study, most interesting case is when both the small parameters are comparable in size, i.e., $\delta/\eps$ tends to a positive constant. In this case, the limiting {\sc sde} for the fluctuations has both a diffusive term due to small noise and an extra drift term which captures the sampling effect. More precisely, in our main results (Theorems \ref{T:LLN-UIT} and \ref{T:fluctuations-R-1-2}), we show that for any natural number $p\ge 1,$
\begin{equation*}
\begin{aligned}
\sup_{t \ge 0}\BE\left[|X_t^{\eps,\delta}-x_t|^p\right] & \le C \eta(\eps,\delta),\\
\sup_{t \ge 0}\BE\left[|Z_t^{\eps,\delta}-Z_t|^p\right] & \le C \theta(\eps,\delta),
\end{aligned}
\end{equation*}
where $C$ is always a \textit{time-independent} positive constant, $\eta(\eps,\delta)$ and $\theta(\eps,\delta)$ are some functions which converge to zero as $\eps,\delta$ vanish and $x_t$ is the solution of \eqref{E:Intro-det-sys}. Here, the stochastic process $Z_t$ solves the {\sc sde}
\begin{equation}\label{E:Intro-limiting-SDE}
Z_t=\int_0^t [Df(x_s)+D\kappa(x_s)]Z_s \thinspace ds - \frac{\cc}{2} \int_0^t D \kappa(x_s)[f(x_s)+\kappa(x_s)] \thinspace ds + \int_0^t \sigma(x_s) dW_s,
\end{equation}
where $\cc \triangleq \lim_{\eps \searrow 0}\frac{\delta}{\eps}$ (see equation \eqref{E:c}). 
As a consequence of the {\sc clt} (Theorem \ref{T:fluctuations-R-1-2}), the stochastic process $X_t^{\eps,\delta}$ can be approximated, uniform-in-time, by the process $x_t+ \eps Z_t$ as $\eps,\delta \searrow 0,$ that is,
\begin{align*}
X_t^{\eps,\delta} &\approx x_t+ \eps Z_t, \qquad t\in [0,\infty).
\end{align*}

Indeed, we make this approximation precise by a first-order perturbation expansion of the stochastic process $X_t^{\eps,\delta}$ (i.e., $X_t^{\eps,\delta}=x_t+ \eps Z_t + o(\eps^2)$) in the powers of small parameter 
 $\eps$ along with the time-independent error bounds on the remainder. Here, the zeroth- and first-order terms are characterized by the {\sc ode} $\frac{dx}{dt}=f(x)+\kappa(x), \thinspace x(0)= x_0$ describing the mean behavior and the {\sc sde} \eqref{E:Intro-limiting-SDE} which captures the fluctuations about the mean.

The results of this paper are novel mainly in terms of the time-independent convergence bounds and identification of an additional drift term in the fluctuation analysis. The fluctuations allow one to write a first-order perturbation expansion of the stochastic process of interest along with the bounds on the remainder uniform-in-time.
 The way at which $\epsilon,\delta\searrow 0$ affects the limiting dynamics in a crucial way. In addition, our proof of the uniform-in-time {\sc clt} shows that what is needed is a control on the behavior of the gradient of the drift of the limiting dynamics (\ref{E:Intro-det-sys}) via Assumption \ref{A:Poly-Int} together with uniform-in-time control of appropriate moments of $X_t^{\eps,\delta}$ and the related uniform-in-time {\sc lln}.

 The difficulty of the mathematical analysis lies on the careful and diligent analysis and appropriate decomposition needed in order to extract the best possible bounds that are also uniform-in-time. The presence of the sampling term complicates the uniform-in-time analysis of the fluctuations.

In recent decades, the asymptotic behavior of several (multi scale) stochastic systems has been explored by many authors in the form of averaging principle/homogenization. This enables one to obtain simple dynamics of the slowly-varying component by averaging over the quickly varying component. For the fundamental work in this direction, we refer to the seminal papers of Khasminskii \cite{khasminskij1968principle,has1966stochastic}. Since then, a great amount of work has been developed in the form of different limit theorems. To mention a few, one can see, for example, \cite{FreidlinSowers,pardoux2001poisson,
pardoux2003poisson,pardoux2005poisson,FW_RPDS,spiliopoulos2014fluctuation,
rockner2021diffusion,rockner2021averaging}, for the averaging principle and/or fluctuation analysis for {\sc sde} with multiple time scales and multiple small parameters in different asymptotic regimes. For the interaction of Large Deviation Principle (which describes the asymptotic behavior of probabilities of rare events in terms of certain rate functions) with stochastic processes, one can see \cite{veretennikov2000large,Spil-AppMathOptim} and the references therein. Averaging and fluctuation analysis for stochastic partial differential equations are investigated in , e.g., \cite{cerrai2009khasminskii,cerrai2009averaging}.

The averaging and fluctuation analysis for multiscale systems have typically been considered over finite time horizons (i.e., on time intervals $[0,T]$ for fixed $T<\infty$). Recently, \cite{crisan2021uniform,fang2020adaptive} provided sufficient conditions for \textit{uniform-in-time} convergence of error of the Euler scheme for a {\sc sde}. Further, a uniform-in-time averaging principle for a slow-fast fully coupled stochastic system, has been explored in \cite{crisan2022poisson}. For the {\sc ldp} over infinite time horizons, we refer to a recent article \cite{budhiraja2024large}.

This paper is organized as follows. In Section \ref{S:Problem-setup}, we present our problem set-up and our main results (Theorems \ref{T:LLN-UIT} and \ref{T:fluctuations-R-1-2}). In the same section, we mention certain regularity assumptions and a sufficient condition, which enable us to accomplish the uniform time analysis. In addition, we present some potential applications of our results and simulation studies that illustrate and complement the theoretical results. Further, Section \ref{S:LLN-Proof} is devoted to the proof of Theorem \ref{T:LLN-UIT} through a series of helpful lemmas. In Section \ref{S:CLT-Proof}, we present the proof of our second main result Theorem \ref{T:fluctuations-R-1-2}. Section \ref{S:Error-Prop-proof} has the proofs of a number of supporting lemmas needed in the proof of Theorem \ref{T:fluctuations-R-1-2} and mainly related to properly controlling the sampling effect.

\subsection*{Notation and conventions}
We list some of the special notations and conventions used throughout this manuscript. The symbol $\triangleq$ is read ``is defined to equal." We denote the set of all positive integers, non-negative integers and real numbers by $\BN$, $\BZ^+,$ and $\BR$, respectively. For any $m \in \BN,$ and a given function $h:\BR \to \BR$, $D^m h(x)$ represents the $m^{\text{th}}$-derivative of $h$ at the point $x\in \BR$; of course, in the higher dimensions $Dh(x)$ stands for the Jacobian matrix. Throughout this manuscript, we use $\pi_\delta(t)$ for $\delta \lfloor {t}/{\delta} \rfloor$ and $\|\cdot\|_{\infty}$ for the sup norm. The set $C_b^1(\BR)$ denotes the collection of all real-valued continuous functions whose first derivatives are bounded. For a random variable $Y$, which has a normal distribution with parameters $\mu$ and $\xi^2,$ we write $Y\sim \mathscr{N}(\mu,\xi^2).$ Throughout the paper, the letter $C$ denotes a positive constant which may depend on various parameters \textit{except for} the time parameter $t$ and the small parameters $\eps$ and $\delta$; the value of $C$ may change from line to line. In the statement of results, the constant $C$ will be denoted with a subscript (like, $C_{\ref{T:LLN-UIT}}$). On many occasions in this paper, we will use the following fundamental inequalities without explicitly mentioning it: for any $n \in \BN, \thinspace p>0$ and positive real numbers $a_1,\cdots a_n,$
\begin{equation*}
\begin{aligned}
a_1 a_2\cdots a_n  \le C(a_1^n+ \cdots
+a_n^n), \quad \text{and}  \quad
(a_1+ \cdots + a_n)^p \le C (a_1^p+ \cdots + a_n^p).
\end{aligned}
\end{equation*}

\section{Problem Statement, Main Results and Application}
\subsection{Problem Statement and Assumptions}\label{S:Problem-setup}
We consider a nonlinear dynamical system affine in control
$\frac{dx}{dt}=f(x)+u, \thinspace x(0)= x_0,$ where $x(t):[0,\infty) \to \BR$ represents the state of the system and $u\in \BR$ is a control input. Throughout this analysis, a state-feedback control law $u=\kappa(x)$ is a priori fixed to get the closed-loop system
\begin{equation}\label{E:det-sys}
\frac{dx}{dt}=f(x)+\kappa(x), \quad x(0)= x_0.
\end{equation}

We now use a sample-and-hold implementation of the control law $u=\kappa(x).$ By this we mean, for a fixed $\delta>0,$ the state $x$ of system \eqref{E:det-sys} is measured at the periodic sampling time instants $k \delta,\thinspace k \in \BZ^+$ and the control is updated according to the feedback control law $u=\kappa(x_{k\delta})$ and is held fixed in equation \eqref{E:det-sys} over the sampling interval $[k \delta, (k+1)\delta)$. In this case, the dynamics of system \eqref{E:det-sys}, in the interval $[k \delta, (k+1)\delta),$ is governed by the differential equation $\frac{dx_t^\delta}{dt}=f(x_t^\delta)+\kappa(x^\delta_{k \delta})$ with the initial conditions $x^\delta_{k \delta}=x^\delta_{k \delta-}$ and $x_{0-}^{\delta}=x_0^\delta=x_0.$ Further, for the analysis purposes, we rewrite the dynamics of $x^\delta_t$ using the
time-discretization function $\pi_\delta(t) \triangleq \delta \lfloor {t}/{\delta} \rfloor, \thinspace t \in [0,\infty)$ to get
\begin{equation}\label{E:Det-sys-delta}
\frac{dx_t^\delta}{dt}=f(x_t^\delta)+\kappa(x^\delta_{\pi_\delta(t)}), \quad x_0^\delta=x_0.
\end{equation}
One can note here that the instantaneous rate $\frac{dx_t^\delta}{dt}$ depends on both the current value $x_t^\delta$ and the most recent sample of the state, that is, $x^\delta_{\pi_\delta(t)}$ and we expect as $\delta\searrow 0$, the dynamics of $x_t^\delta$ converges to $x_t$ solving \eqref{E:det-sys}.

We would now like to explore the situation when the system \eqref{E:Det-sys-delta} is subjected to a small white-noise effect. To make things precise, we assume that the process $W=\{W_t:t\ge 0\}$, defined on a probability space $(\Omega, \mathscr{F},\BP),$ represents a one-dimensional Brownian motion and the parameter $\eps>0$ corresponds to the size of noise. The resulting dynamics of system \eqref{E:Det-sys-delta} is now described by the stochastic process $X^{\eps,\delta}=\{X_t^{\eps,\delta}:t \ge 0\}$ which solves the {\sc sde}
\begin{equation}\label{E:W-S-SDE-With-pert}
dX_t^{\eps,\delta}=\left\{f(X_t^{\eps,\delta})+\kappa(X_{\pi_\delta(t)}^{\eps,\delta})\right\}dt+\eps \sigma(X_t^{\eps,\delta})dW_t, \quad X_0^{\eps,\delta}=x_0.
\end{equation}

Throughout this paper, the functions $f:\BR \to \BR,$ $\kappa: \BR \to \BR$ and $\sigma: \BR \to \BR$ in model \eqref{E:W-S-SDE-With-pert} are assumed to satisfy the Assumptions \ref{A:Poly-Lip}, \ref{A:Poly-growth}, \ref{A:Poly-Int} below and that they yield a unique and, at least, weak solution to (\ref{E:W-S-SDE-With-pert}).

\begin{assump}\label{A:Poly-Lip}
The mapping $f$ satisfies local polynomial growth Lipschitz continuity condition and the functions $\kappa, \sigma$ are assumed globally Lipschitz continuous, i.e.,
for all $x,y \in \BR,$ there exist positive numbers $q, \xi$ and  $\mu$ such that
\begin{equation}\label{E:Poly-Dissp-Growth}
\begin{aligned}
|f(x)-f(y)| &\le [\xi(|x|^q+|y|^q)+\mu]|x-y|,\\
|\kappa(x)-\kappa(y)| &\le L_{\kappa} |x-y|,\nonumber\\
|\sigma(x)-\sigma(y)| &\le L_{\sigma} |x-y|,
\end{aligned}
\end{equation}
where $L_{\kappa},L_{\sigma}$ are the global Lipschitz constants for $\kappa,\sigma$ respectively. In addition, we assume $f$ satisfies a contractive Lipschitz continuity condition. In particular,
we assume that there is $\lambda>0$ such that for $x,y\in\BR$, we have
\begin{equation}\label{E:Poly-Diss-Lip}
\begin{aligned}
(x-y)\cdot (f(x)-f(y))  \le -\lambda |x-y|^2.
\end{aligned}
\end{equation}
Furthermore, we assume that the constants $\lambda,L_{\kappa}$ are such that $\frac{\lambda}{2}-L_{\kappa}>0$.
\end{assump}
An important consequence of (\ref{E:Poly-Diss-Lip}) is that there exist positive numbers $\alpha$, $\beta$ so that
\begin{align}
x\cdot f(x) & \le -\alpha |x|^2+ \beta \quad \text{for}\quad x\in\BR.\label{E:Poly-diss-f}
\end{align}

\begin{assump}\label{A:Poly-growth}
The functions $\kappa,\sigma$ are uniformly bounded, i.e.,
there exists $\gamma<\infty$ such that
\begin{equation*}
\begin{aligned}
|\sigma(x)| +|\kappa(x)| &\le \gamma, \quad x \in \BR.
\end{aligned}
\end{equation*}
\end{assump}

\begin{assump}\label{A:Poly-Int}
We assume that $f,\kappa$ are twice differentiable and that for $m \in \{1,2\},$ there exists a time-independent positive constant $C$ such that
\begin{equation*}
\sup_{t\ge 0}\int_0^t e^{\int_s^t m[Df(x_u)+D\kappa (x_u)]\thinspace du}\thinspace ds \le C.
\end{equation*}
The second derivative of $\kappa$ is assumed to be uniformly bounded whereas the second derivative of $f$ is assumed to grow polynomially.
\end{assump}

A few remarks on Assumptions \ref{A:Poly-Lip}, \ref{A:Poly-growth} and \ref{A:Poly-Int} are in order.
\begin{remark}
Assumptions \ref{A:Poly-Lip}, \ref{A:Poly-growth} and \ref{A:Poly-Int} on $f$ are satisfied, for example, by the functions $f(x)=-\lambda x$, $f(x)=-x^3-\lambda x$ and bounded, differentiable functions $\kappa$ with sufficiently small Lipschitz constant.

We also note that the contractive Lipschitz condition is a consequence of the one-sided Lipschitz condition with negative ``one-sided Lipschitz" constant. In particular, the function $f$ such that for  $x>y$ the relation holds
\begin{align*}
f(x)-f(y)&\leq -\lambda (x-y),
\end{align*}
will also satisfy (\ref{E:Poly-Diss-Lip}). Note that for such $f$ that are also differentiable, we have that $\sup_{x}f'(x)=-\lambda$.
\end{remark}

\begin{remark}
Assumption (\ref{E:Poly-Diss-Lip})  is used in the proof of Theorem \ref{T:LLN-UIT} (which proves the uniform convergence of $X^{\eps,\delta}_{t}$ to $x_{t}$) and is needed in order to control the magnitude  of the term $|X^{\eps,\delta}_{t}-x_{t}|$ for any given $t\in [0,\infty)$.
\begin{remark}\label{R:Dissipativity-condtion}
Relation (\ref{E:Poly-diss-f}) is used in Lemma \ref{L:Poly-p-Moment} to prove the uniform moment bound for $X^{\eps,\delta}_{t}$ and the uniform bound for $x_{t}$. We note that a simple inspection of the proof of Lemma \ref{L:Poly-p-Moment} shows that the claim holds if we only assume (\ref{E:Poly-diss-f}) to be true for $|x|>R$ for some $R>0$. We chose to use (\ref{E:Poly-diss-f}) as a consequence of (\ref{E:Poly-Diss-Lip}) instead of using this more general condition because, despite our efforts, we could not relax accordingly Assumption (\ref{E:Poly-Diss-Lip}) that is used for Theorem \ref{T:LLN-UIT}. This is due to the fact that in Theorem \ref{T:LLN-UIT}, we need to appropriately control the term $(X^{\eps,\delta}_{t}-x_{t})[f(X^{\eps,\delta}_{t})-f(x_t)]$ for all $t\in[0,\infty)$ and not to just control the behavior for large $|X^{\eps,\delta}_{t}|$. In particular, bounding moments of $|X^{\eps,\delta}_{t}-x_{t}|$ is not enough, we need those moments to go to zero as $\eps,\delta\searrow 0$.
\end{remark}

However, in Subsection \ref{S:Simulation}, we demonstrate our principal results via simulation studies also on examples that do not necessarily satisfy all of the assumptions, demonstrating that our theoretical result is expected to hold broader.
\end{remark}

\begin{remark}
Assumption \ref{A:Poly-Int} is only used in the proof of the uniform-in-time {\sc clt}, Theorem \ref{T:fluctuations-R-1-2}. The proof of Theorem \ref{T:fluctuations-R-1-2} makes it clear, that besides Assumption \ref{A:Poly-Int}, one mainly needs to have the uniform moment bounds and uniform law of large numbers result of Lemma \ref{L:Poly-p-Moment} and Theorem \ref{T:LLN-UIT} respectively and then the proof of the uniform-in-time {\sc clt} goes through. This means that the uniform-in-time {\sc clt} will hold under a proper control of the derivative of the drift functions via Assumption \ref{A:Poly-Int} and any assumptions that guarantee the validity of Lemma \ref{L:Poly-p-Moment} and Theorem \ref{T:LLN-UIT}. The assumption on the existence of the second derivatives of $f$ and $\kappa$ stems from the need to control the fluctuations in the technical lemmas of Subsection \ref{SS:FluctuationsControl}, which we accomplish through appropriate Taylor expansions.
\end{remark}

\subsection{Main Results}

We now state our first main result.
\begin{theorem}(Law of Large Numbers Type Result)\label{T:LLN-UIT}
Let $x_{t}$ and $X_{t}^{\varepsilon,\delta}$ be the solutions of \eqref{E:det-sys} and \eqref{E:W-S-SDE-With-pert}, respectively satisfying Assumptions  \ref{A:Poly-Lip}, \ref{A:Poly-growth}. Then, for any $p \in \BN$ and for all $\eps,\delta>0$ sufficiently small, there exists a time-independent positive constant $C_{\ref{T:LLN-UIT}}$ such that
\begin{equation}\label{E:LLN-SM}
\sup_{t \ge 0}\BE \left[|X_t^{\eps,\delta}-x_t|^p\right] \le C_{\ref{T:LLN-UIT}}(\delta^p+\eps^p+{\delta}^{\frac{p}{2}}\eps^p).
\end{equation}
\end{theorem}

This result essentially gives the rate of convergence of the stochastic process $X_t^{\eps,\delta}$ to its deterministic counterpart $x_t$ uniformly in time as $\eps,\delta \searrow 0$. This theorem can be interpreted as a Law of Large Numbers ({\sc lln}) type result. We present the proof of Theorem \ref{T:LLN-UIT} in Section \ref{S:LLN-Proof}.

We next explore the fluctuation analysis of $X_t^{\eps,\delta}$ uniform-in-time about its mean $x_t$. Here, the fluctuation behavior is found to vary, depending on the relative rates at which the two small parameters $\eps,\delta$ tend to zero. To make this precise, we define two different regimes:
\begin{equation}\label{E:c}
\cc \triangleq \lim_{\eps \searrow 0}\delta_\eps/\eps \begin{cases}=0 & \text{Regime 1,}\\ \in (0,\infty) & \text{Regime 2,}\end{cases}
\end{equation}
where, we assume $\delta=\delta_{\eps}$ and $\lim_{\eps \searrow 0}\delta_\eps/\eps$ exists in $[0, \infty).$ For Regimes 1 and 2, we consider the rescaled fluctuation process
\begin{equation}\label{E:Fluc-Proc}
Z^{\eps,\delta}_t \triangleq \frac{X^{\eps,\delta}_t - x_t}{\eps}.
\end{equation}
Of course, the fluctuation process  $Z_t^{\eps,\delta}$ solves the stochastic integral equation \eqref{E:Fluctuation-R-1-2}. We now present our second main result which can be interpreted as a Central Limit Theorem ({\sc clt}) type result. This result focuses on the limiting behavior, uniform-in-time, of the process $Z_t^{\eps,\delta}$ as $\eps,\delta$ approach to zero.
\begin{theorem}(Central Limit Theorem Type Result)\label{T:fluctuations-R-1-2}
Let $x_{t}$ and $X_{t}^{\varepsilon,\delta}$ be the solutions of \eqref{E:det-sys} and \eqref{E:W-S-SDE-With-pert}, respectively satisfying Assumptions  \ref{A:Poly-Lip}, \ref{A:Poly-growth}, \ref{A:Poly-Int} and the fluctuation process $Z^{\eps,\delta}_t$ be defined in equation \eqref{E:Fluc-Proc}. Suppose that we are in Regime $i \in \{1,2\}$, i.e., $\lim_{\eps \searrow 0}\delta_\eps/\eps = \cc \in [0,\infty)$. 
Let $Z=\{Z_t: t \ge 0\}$ be the unique solution of
\begin{equation}\label{E:Limiting-SDE-Fluctuations}
 Z_t=\int_0^t [Df(x_s)+D\kappa(x_s)]Z_s \thinspace ds - \frac{\cc}{2} \int_0^t D \kappa(x_s)[f(x_s)+\kappa(x_s)] \thinspace ds + \int_0^t \sigma(x_s) dW_s.
\end{equation}
Then, for any integer $p \ge 1$, there exists a time-independent positive constant $C_{\ref{T:fluctuations-R-1-2}}$ such that for all sufficiently small $\eps,\delta>0$, we  have
\begin{multline}\label{E:FCLT}
\sup_{ t \ge 0}\BE\left[|Z^{\eps,\delta}_t - Z_t|^p\right] =
\frac{1}{\eps^p}\sup_{ t \ge 0}\BE\left[|X^{\eps,\delta}_t - x_t - \eps Z_t |^p \right] \le \\
C_{\ref{T:fluctuations-R-1-2}}\left[\frac{\delta^{2p}}{\eps^p}+\left|\frac{\delta}{\eps}-\cc \right|^p+\left(\frac{\delta^p}{\eps^p}+1 \right)(\delta^p+\eps^p+{\delta}^{\frac{p}{2}}\eps^p)^{\frac{1}{2}}+\delta^{p/2}+\eps^p\right].
\end{multline}
\end{theorem}

\begin{remark}(Interpretation)
The result (Theorem \ref{T:fluctuations-R-1-2}) can be interpreted as follows: The expression $\sup_{ t \ge 0}\BE\left[|X^{\eps,\delta}_t - x_t - \eps Z_t|^p\right]$ tends to zero faster than $\eps^p$ does; hence $\sup_{ t \ge 0}\BE\left[|X^{\eps,\delta}_t - x_t - \eps Z_t|^p\right]=o(\eps^p)$. The latter enables us to approximate in the limit the non-Markovian (due to a memory of length $\delta$) process $X_t^{\eps,\delta}$ by the Markov process $x_t+\eps Z_t.$
\end{remark}

\subsection{Assumptions for a generalization of our model}
In this section, we provide the conditions under which our results hold for a generalized version of equation \eqref{E:W-S-SDE-With-pert} in higher dimensions. We consider the {\sc sde} and {\sc ode}
\begin{equation}\label{E:NL-SDE-gen}
 dX_t^{\eps,\delta}=\left\{f(X_t^{\eps,\delta})+g(X_t^{\eps,\delta})\kappa(X_{\pi_\delta(t)}^{\eps,\delta})\right\}dt+\eps \sigma(X_t^{\eps,\delta})dW_t,
 \end{equation}
 \begin{equation}\label{E:NL-ODE-gen}
 \frac{dx_t}{dt}=f(x_t)+g(x_t)\kappa(x_t),
 \end{equation}
where, for $d,n \ge 1$, $f:\BR^n \to \BR^n, \thinspace g: \BR^n \to \BR^{n \times d}, \thinspace \kappa: \BR^n \to \BR^d$ and $\sigma: \BR^n \to \BR^{n \times n}$ are certain regular functions and $W$ is an $n$-dimensional Brownian motion.

The matrix-valued function $\Phi(t)\triangleq e^{\int_0^t [Df(x_s)+D(g\kappa)(x_{s})]\thinspace ds}$  satisfies the matrix differential equations
\begin{equation}\label{E:matrix-eq}
\begin{aligned}
\frac{d}{dt}\Phi(t)&=[Df(x_t)+D(g\kappa)(x_{t})]\Phi(t), \quad \Phi(0)=I.
\end{aligned}
\end{equation}

In equation \eqref{E:matrix-eq}, $x_t$ solves {\sc ode} \eqref{E:NL-ODE-gen}, $I$ represents the identity matrix, and $D(g \kappa)(x_t) \triangleq g(x_t)D \kappa(x_t)+ \sum_{i=1}^d Dg_i(x_t) \cdot \kappa_i(x_t)$ for the column vectors $g_i \in \BR^n$ of matrix $g$ and the entries $\kappa_i \in \BR$ of matrix $\kappa$, $i=1, \cdots, d$.


We now provide the sufficient conditions below under which one can get Theorems \ref{T:LLN-UIT} and \ref{T:fluctuations-R-1-2} for model \eqref{E:NL-SDE-gen}, and these conditions can be thought as corresponding to Assumptions \ref{A:Poly-Lip}, \ref{A:Poly-growth} and \ref{A:Poly-Int}. In the following assumptions, let $\langle \cdot , \cdot \rangle$ and $|\cdot|$ represent the standard inner product and Euclidean norm of vectors, respectively. For any matrix $A= [a_{ij}]\in \BR^{n \times n}$, $|A|_F \triangleq \sqrt{\sum_{i,j=1}^na_{ij}^2}$ represents the Frobenius norm. \\
\noindent

\textbf{Assumption (A1)$'$.} The mapping $f$ satisfies the local polynomial growth Lipschitz continuity condition and the functions $\kappa, \sigma$ and $g$ are assumed globally Lipschitz continuous, i.e., for all $x,y \in \BR^n,$ there exist positive numbers $q, \xi, L$ and  $\mu$ such that
\begin{align}
\left|f(x)-f(y)\right| &\le \left[\xi(|x|^q+|y|^q)+\mu\right]|x-y|,\nonumber\\
|\kappa(x)-\kappa(y)|+|\sigma(x)-\sigma(y)|_F+|g(x)-g(y)|_F &\le L |x-y|.\nonumber
\end{align}
Furthermore, we assume $f$ satisfies a contractive Lipschitz continuity. In particular, we assume that there exists  $0<\lambda<\infty$  such that for all $x,y\in\BR^{n}$, we have
\begin{align*}
\langle x-y, f(x)-f(y) \rangle & \le -\lambda |x-y|^2.
\end{align*}

With $L_{g\kappa }$ denoting the Lipschitz constant associated to the function $g\kappa $, we assume that $\frac{\lambda}{2}-L_{g\kappa }>0$.
As in one-dimensional case, a consequence of the contractive Lipschitz conditions is that there exist positive numbers $\alpha$ and $\beta$ such that
\begin{equation*}
\begin{aligned}
\langle x, f(x) \rangle & \le -\alpha |x|^2+ \beta \quad \text{for}\quad x\in\BR^{n}.
\end{aligned}
\end{equation*}

\noindent
\textbf{Assumption (A2)$'$.}
The functions $\kappa,\sigma$ and $g$ are uniformly bounded, i.e., there exists $\gamma<\infty$ such that
\begin{equation*}
\begin{aligned}
|\kappa(x)|+ |\sigma\sigma^{\top}(x)|_F + |g(x)|_F &\le \gamma, \quad x \in \BR^n.
\end{aligned}
\end{equation*}

\noindent
\textbf{Assumption (A3)$'$.}
For  $m \in \{1,2\},$ there exists a time-independent positive constant $C$ such that
\begin{equation*}
\sup_{t\ge 0}\int_0^t \left[|\Phi(t)\Phi(s)^{-1}|_F\right]^m ds \le C,
\end{equation*}
where, $\Phi(t)\triangleq e^{\int_0^t [Df(x_s)+D(g\kappa)(x_{s})]\thinspace ds}$ is the solution of equation \eqref{E:matrix-eq}.

 Since the idea of the proof of Theorems \ref{T:LLN-UIT} and \ref{T:fluctuations-R-1-2} with Assumptions (A1)$'$, (A2)$'$, (A3)$'$  is same as with Assumptions \ref{A:Poly-Lip}, \ref{A:Poly-growth} and \ref{A:Poly-Int},  for the sake of presentation, we present the proof only in one-dimensional case with $g=1$ in \eqref{E:NL-SDE-gen}.

\subsection{An application of Theorem \ref{T:fluctuations-R-1-2}} We provide an application of our main result (Theorem \ref{T:fluctuations-R-1-2}) in Corollary \ref{C:Theorem-App} below. The latter essentially states that the behavior of the stochastic process $X_t^{\eps,\delta}$ is close (uniform-in-time) to a Gaussian process $x_t+\eps Z_t$ with a specified  mean and variance (see equation \eqref{E:Mean-Var}).
\begin{corollary}\label{C:Theorem-App}
Let $X_{t}^{\varepsilon,\delta}$ be the solution of \eqref{E:W-S-SDE-With-pert} and $V_t^\eps \triangleq x_t+\eps Z_t$, where $x_t$ and $Z_t$ solve \eqref{E:det-sys} and \eqref{E:Limiting-SDE-Fluctuations}, respectively. Then, for any $\varphi \in C_b^1(\BR),$ and integer $p \ge 1,$ there exists a positive constant $C_{\ref{C:Theorem-App}}$ (time-independent) such that
\begin{equation*}
\sup_{t\ge 0}\BE\left[|\varphi(X_t^{\eps,\delta})-\varphi(V_t^{\eps}) |^p\right]  \le
 \eps^p C_{\ref{C:Theorem-App}}\left[\frac{\delta^{2p}}{\eps^p}+\left|\frac{\delta}{\eps}-\cc \right|^p+\left(\frac{\delta^p}{\eps^p}+1 \right)(\delta^p+\eps^p+{\delta}^{\frac{p}{2}}\eps^p)^{\frac{1}{2}}+\delta^{p/2}+\eps^p\right].
\end{equation*}
The stochastic process $V_t^\eps$ is a Gaussian process with mean $\mu_t$ and variance $\xi_t^2$ given by
\begin{equation}\label{E:Mean-Var}
\begin{aligned}
\mu_t&=x_t-\frac{\cc \eps}{2}\int_0^t D \kappa(x_s) e^{\int_s^t[Df(x_u)+D\kappa(x_u)]\thinspace du}[f(x_s)+\kappa(x_s)]\thinspace ds, \\
 \xi_t^2 &=\eps^2\int_0^t e^{2\int_s^t[Df(x_u)+D\kappa(x_u)]\thinspace du}\sigma^2(x_s)\thinspace ds.
\end{aligned}
\end{equation}
\end{corollary}

\begin{proof}
Using It\^{o}'s formula in equation \eqref{E:Limiting-SDE-Fluctuations} (i.e., a limiting {\sc sde} for the fluctuation process $Z_t^{\eps,\delta}$), we obtain
\begin{equation*}
Z_t=-\frac{\cc}{2}\int_0^t D \kappa(x_s) e^{\int_s^t[Df(x_u)+D\kappa(x_u)]\thinspace du}[f(x_s)+\kappa(x_s)]\thinspace ds + \int_0^t e^{\int_s^t[Df(x_u)+D\kappa(x_u)]\thinspace du}\sigma(x_s)\thinspace dW_s.
\end{equation*}
Now, let $V_t^\eps \triangleq x_t+\eps Z_t$, then, $V_t^\eps \sim \mathscr{N}(\mu_t,\xi_t^2),$ where, the expressions for the mean $\mu_t$ and variance $\xi_t^2$ are defined as follows:
\begin{equation*}
\mu_t=x_t-\frac{\cc \eps}{2}\int_0^t D \kappa(x_s) e^{\int_s^t[Df(x_u)+D\kappa(x_u)]\thinspace du}[f(x_s)+\kappa(x_s)]\thinspace ds, \thinspace  \xi_t^2=\eps^2\int_0^t e^{2\int_s^t[Df(x_u)+D\kappa(x_u)]\thinspace du}\sigma^2(x_s)\thinspace ds.
\end{equation*}
Next, for any $\varphi \in C_b^1(\BR),$ a simple algebra followed by Taylor's theorem yields
\begin{equation*}
\begin{aligned}
\varphi(X_t^{\eps,\delta})=\varphi\left(\frac{X_t^{\eps,\delta}-x_t}{\eps}\eps+x_t\right)=\varphi \left(\eps Z_t^{\eps,\delta}-\eps Z_t + \eps Z_t +x_t \right)
 = \varphi(\eps Z_t +x_t)+ \eps\varphi'(z)[Z_t^{\eps,\delta}-Z_t],
\end{aligned}
\end{equation*}
where $z \in \BR$ is a point lying on the line segment joining $X_t^{\eps,\delta}$ and $x_t+\eps Z_t$. Hence, for any integer $p\ge 1,$ we have
\begin{equation*}
\left|\varphi(X_t^{\eps,\delta})-\varphi(V_t^{\eps})\right|^p \le \eps^p|\varphi'(z)|^p\left|Z_t^{\eps,\delta}-Z_t \right|^p.
\end{equation*}
Finally, using the boundedness of $\varphi'$ followed by Theorem \ref{T:fluctuations-R-1-2}, one can obtain the required result.
\end{proof}

\subsection{Numerical Examples and Simulation}\label{S:Simulation}
This section is devoted to the numerical illustration of the application of our theoretical result (i.e., Corollary \ref{C:Theorem-App}) in the context of a few simple examples.
In Example 1, the functions $f$ and $\kappa$ satisfy the specified assumptions, however, Examples 2, 3 and 4 demonstrate that our main results can be expected to hold when $f$ satisfies the dissipative condition for large values of $x$, or the control function $\kappa$ is not necessarily a bounded function. In this paper, we provide conditions for the generic model under which we could prove the uniform-in-time results. Examples 2, 3 and 4 demonstrate that for concrete models with special structure, one can go beyond the generic assumptions of this paper.
\begin{remark}
It is important to note that our main results are derived under certain assumptions including that the function $f$ satisfies contractive Lipschitz continuity (Equation \eqref{E:Poly-Diss-Lip}) and the control function $\kappa$ is bounded. The significance of these conditions is as follows. Due to the contractive Lipschitz continuity of $f$, in Theorem \ref{T:LLN-UIT}, we are able to demonstrate that the term $|X_s^{\eps,\delta}-x_s|^p$ tends to zero via handling a challenging term $(X^{\eps,\delta}_{t}-x_{t})[f(X^{\eps,\delta}_{t})-f(x_t)]$. We refer to Remark \ref{R:Dissipativity-condtion} for more details. Next,
the boundedness of $\kappa$ is crucial as it is used (for the first-time) in the proof of Lemma 3.1. It helps to control the product involving the sampling term $(X_t^{\eps,\delta})^{p-1}|\kappa(X_{\pi_\delta(t)}^{\eps,\delta}|,\thinspace t \in [0, \infty)$ which is difficult to handle without assuming the boundedness of $\kappa$ in {\sc uit} setting.
\end{remark}

\subsection*{Example 1}
We consider the non-linear system
\begin{equation}\label{E:example-nonlinear-2}
dX_t^{\eps,\delta}=[-(X_t^{\eps,\delta})^3 - X_t^{\eps,\delta}]dt - \frac{1}{1+e^{-X_{\pi_\delta(t)}^{\eps,\delta}}} dt + \eps dW_t,
\end{equation}
where it is easy to check that the functions $f(x)= -x^3-x,\thinspace \sigma(x)=1$ and $\kappa(x) = -\frac{1}{1+e^{-x}}$ satisfy Assumptions \ref{A:Poly-Lip}, \ref{A:Poly-growth} and \ref{A:Poly-Int}. For the Monte Carlo simulation of the quantity $\sup_{t\ge 0}\left|\BE\varphi(X_t^{\eps,\delta})-\BE\varphi(V_t^{\eps}) \right|^p$ over the $300$ sample paths of Brownian motion, we fix the function $\varphi(x)=x$ and the parameters $T=2^7, X_0^{\eps,\delta}=-0.07,1.5$ and $\delta=2\eps.$ Figure \ref{fig:UIT-Poly-21-22} and Table \ref{Table-UIT-Poly-21-22} show that the absolute error decreases as the values of $\eps$ get small.
\begin{figure}
\centering
\begin{subfigure}{.5\textwidth}
  \centering
  \includegraphics[width=.97\linewidth]{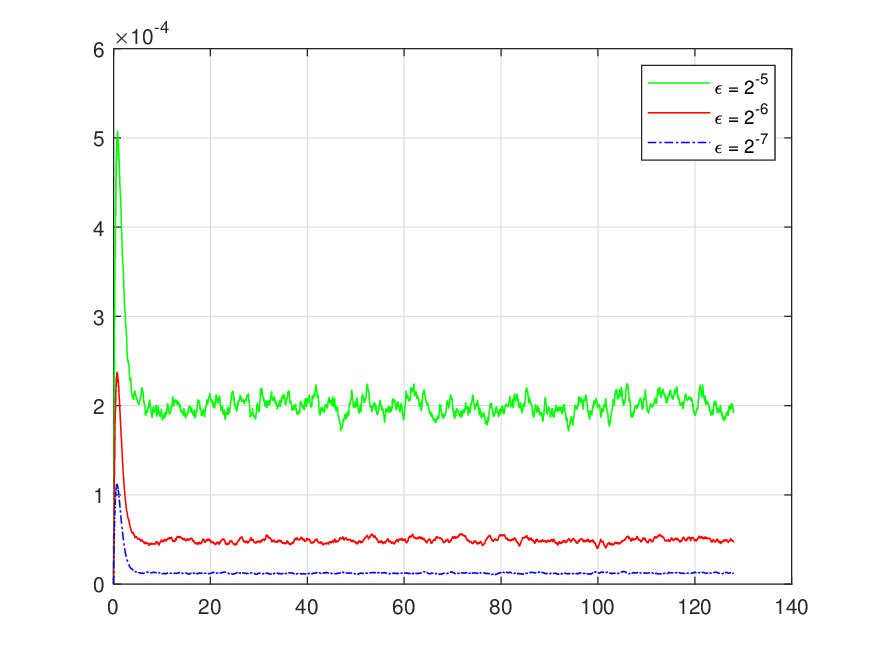}
  \caption{$X_0^{\eps,\delta}=-0.07$}
  \label{fig:sub1}
\end{subfigure}%
\begin{subfigure}{.5\textwidth}
  \centering
  \includegraphics[width=.97\linewidth]{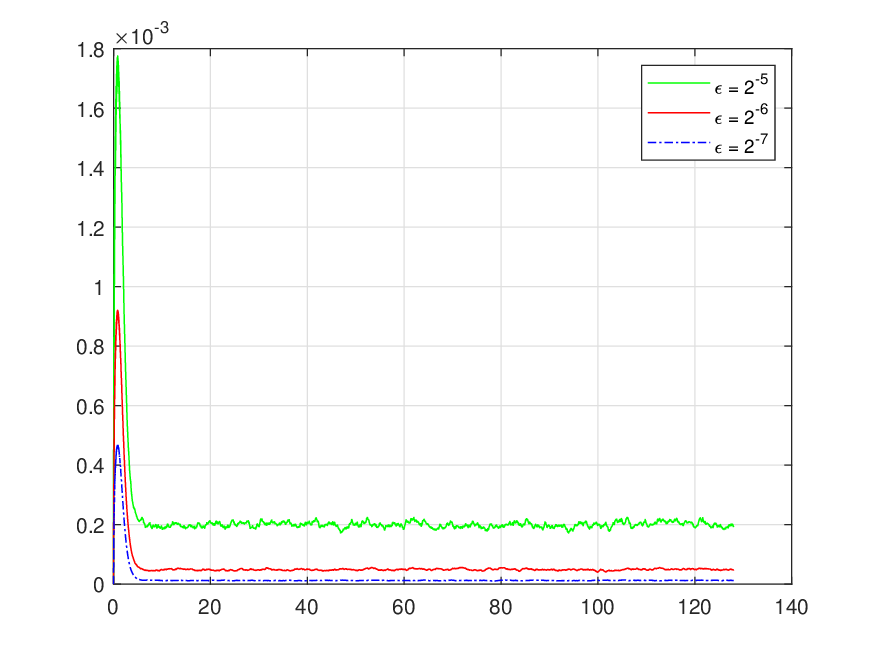}
  \caption{$X_0^{\eps,\delta}=1.5$}
  \label{fig:sub2}
\end{subfigure}
\caption{For the non-linear model \eqref{E:example-nonlinear-2}, varying effect of $\BE[X_T^{\eps,\delta}-x_T-\eps Z_T]$ over $300$ sample paths with different values of $\eps$ and two initial conditions $X_0^{\eps,\delta}=-0.07,1.5$. Here, the error $\BE[X_T^{\eps,\delta}-x_T-\eps Z_T]$ and time are represented on the vertical and horizontal axes, respectively.}
\label{fig:UIT-Poly-21-22}
\end{figure}

\begin{table}[!h]
\begin{center}
\caption{The absolute maximum and minimum values of $\text{Error} \triangleq \BE[X_T^{\eps,\delta}-x_T-\eps Z_T]$ over $300$ sample paths decrease as the values of $\eps$ decrease.}
\label{Table-UIT-Poly-21-22}
\begin{tabular}{ |p{2cm}|p{2cm}||p{3cm}|p{3cm}|  }
 \hline
 \multicolumn{4}{|c|}{Initial condition $X_0^{\eps,\delta}=-0.07$} \\
 \hline
 $\eps$ & $\delta=2\eps$ & $|$Maximum Error$|$ & $|$Minimum Error$|$\\
 \hline
 $2^{-5}$ & $2^{-4}$   & $5.0744 \times 10^{-4}$    & $8.9676 \times 10^{-5}$ \\
 $2^{-6}$ & $2^{-5}$    &   $2.3710 \times 10^{-4}$  & $2.5126 \times 10^{-5}$  \\
 $2^{-7}$  & $2^{-6}$      &   $1.1243 \times 10^{-4}$  & $6.1564 \times 10^{-6}$  \\
  \hline
  \multicolumn{4}{|c|}{Initial condition $X_0^{\eps,\delta}=1.5$} \\
 \hline
 $2^{-5}$ & $2^{-4}$   & $0.0018$    & $1.7179 \times 10^{-4}$ \\
 $2^{-6}$ & $2^{-5}$    &   $9.2053 \times 10^{-4}$  & $3.9899 \times 10^{-5}$  \\
 $2^{-7}$  & $2^{-6}$      &   $4.6849 \times 10^{-4}$  & $1.0321 \times 10^{-5}$  \\
 \hline
\end{tabular}
\end{center}
\end{table}

\subsection*{Example 2}
We consider the {\sc sde}
\begin{equation}\label{E:example-nonlinear-1}
dX_t^{\eps,\delta}=[-(X_t^{\eps,\delta})^3+X_t^{\eps,\delta}]dt - \frac{1}{1+e^{-X_{\pi_\delta(t)}^{\eps,\delta}}} dt + \eps dW_t.
\end{equation}
In this case, we illustrate the varying effect of $\eps$ (noise size) on the quantity $\sup_{t\ge 0}\left|\BE\varphi(X_t^{\eps,\delta})-\BE\varphi(V_t^{\eps}) \right|^p$ for the particular function $\varphi(x)=x, \thinspace p=1$ using Monte Carlo simulation over $300$ sample paths of Brownian motion. The process $V_t^\eps$ is defined as $x_t+\eps Z_t.$ We fix the parameters $T=2^7, X_0^{\eps,\delta}=-0.07,1.5$ and $\delta=2\eps.$ For these fixed parameters, Figure \ref{fig:UIT-Poly-11-12} and Table \ref{Table-UIT-Poly-11-12} demonstrate that as the values of $\eps$ decrease, their corresponding errors (which is defined as the mean of $[X_T^{\eps,\delta}-x_T-\eps Z_T]$ over the specified number of sample paths) also decrease.
\begin{figure}
\centering
\begin{subfigure}{.5\textwidth}
  \centering
  \includegraphics[width=.97\linewidth]{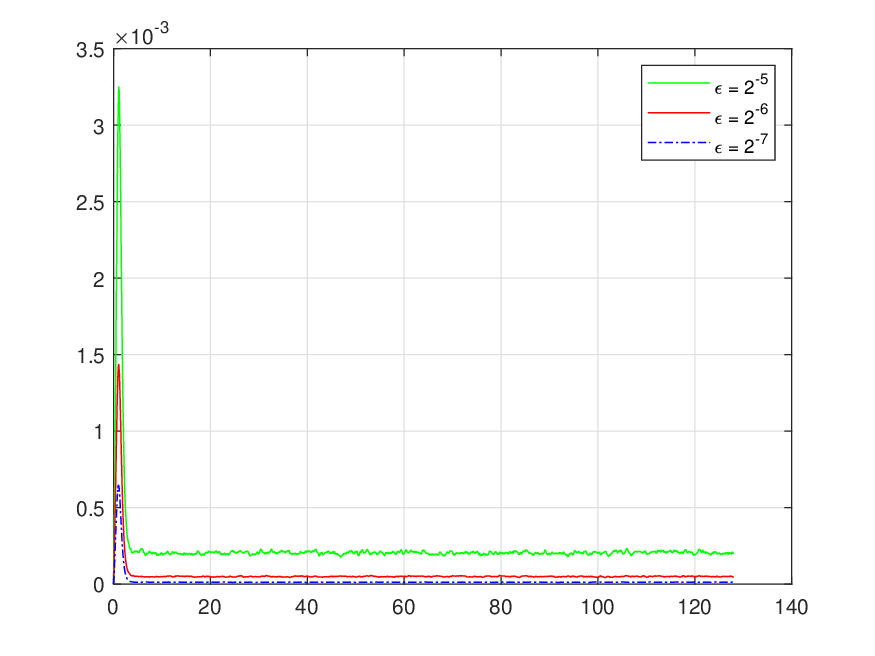}
  \caption{$X_0^{\eps,\delta}=-0.07$}
  \label{fig:sub1}
\end{subfigure}%
\begin{subfigure}{.5\textwidth}
  \centering
  \includegraphics[width=.97\linewidth]{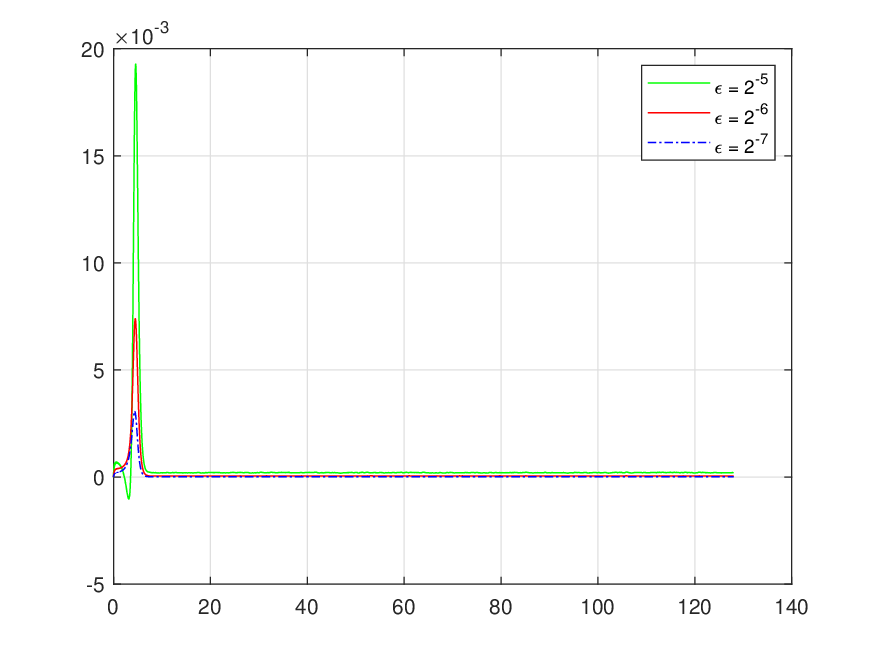}
  \caption{$X_0^{\eps,\delta}=1.5$}
  \label{fig:sub2}
\end{subfigure}
\caption{For the non-linear model \eqref{E:example-nonlinear-1}, varying effect of $\BE[X_T^{\eps,\delta}-x_T-\eps Z_T]$ over $300$ sample paths with different values of $\eps$ and two initial conditions $X_0^{\eps,\delta}=-0.07,1.5$. 
}
\label{fig:UIT-Poly-11-12}
\end{figure}

\begin{table}[!h]
\begin{center}
\caption{The absolute maximum and minimum values of $\text{Error} \triangleq \BE[X_T^{\eps,\delta}-x_T-\eps Z_T]$ over $300$ sample paths decrease as the values of $\eps$ decrease.}
\label{Table-UIT-Poly-11-12}
\begin{tabular}{ |p{2cm}|p{2cm}||p{3cm}|p{3cm}|  }
 \hline
 \multicolumn{4}{|c|}{Initial condition $X_0^{\eps,\delta}=-0.07$} \\
 \hline
 $\eps$ & $\delta=2\eps$ & $|$Maximum Error$|$ & $|$Minimum Error$|$\\
 \hline
 $2^{-5}$ & $2^{-4}$   & $0.0032$    & $1.346 \times 10^{-4}$ \\
 $2^{-6}$ & $2^{-5}$    &   $0.0014$  & $3.366 \times 10^{-5}$  \\
 $2^{-7}$  & $2^{-6}$      &   $6.496 \times 10^{-4}$  & $8.4150 \times 10^{-6}$  \\
  \hline
  \multicolumn{4}{|c|}{Initial condition $X_0^{\eps,\delta}=1.5$} \\
 \hline
 $2^{-5}$ & $2^{-4}$   & $0.0193$    & $0.0010$ \\
 $2^{-6}$ & $2^{-5}$    &   $0.0074$  & $4.2751 \times 10^{-5}$  \\
 $2^{-7}$  & $2^{-6}$      &   $0.0031$  & $1.0880 \times 10^{-5}$  \\
 \hline
\end{tabular}
\end{center}
\end{table}

\subsection*{Example 3}
We consider the {\sc sde}
\begin{equation}\label{E:example-linear-1}
dX_t^{\eps,\delta}=-3X_t^{\eps,\delta}dt -0.3166 X_{\pi_\delta(t)}^{\eps,\delta}dt + \eps dW_t.
\end{equation}
Of course, the functions $f(x)=-3x, \thinspace \sigma(x)=1$ satisfy (\ref{E:Poly-Diss-Lip}) and Assumption \ref{A:Poly-growth} and the function $\kappa(x)=-0.3166x$ has linear growth (not necessarily bounded). In this case, we illustrate the varying effect of $\eps$ (noise size) on the quantity $\sup_{t\ge 0}\left|\BE\varphi(X_t^{\eps,\delta})-\BE\varphi(V_t^{\eps}) \right|^p$ for the particular function $\varphi(x)=x, \thinspace p=1$ using Monte Carlo simulation over $300$ sample paths of Brownian motion. The process $V_t^\eps$ is defined as $x_t+\eps Z_t.$ We fix the parameters $T=2^7, X_0^{\eps,\delta}=0.1$ and $\delta=2\eps.$ For these fixed parameters, Figure \ref{F:Error-linear} and Table \ref{Table-linear} demonstrate that as the values of $\eps$ decrease, their corresponding errors (which is defined as the mean of $[X_T^{\eps,\delta}-x_T-\eps Z_T]$ over the specified number of sample paths) also decrease.

\begin{figure}[h!]
\begin{center}
\includegraphics[height=6cm,width=9.6cm]{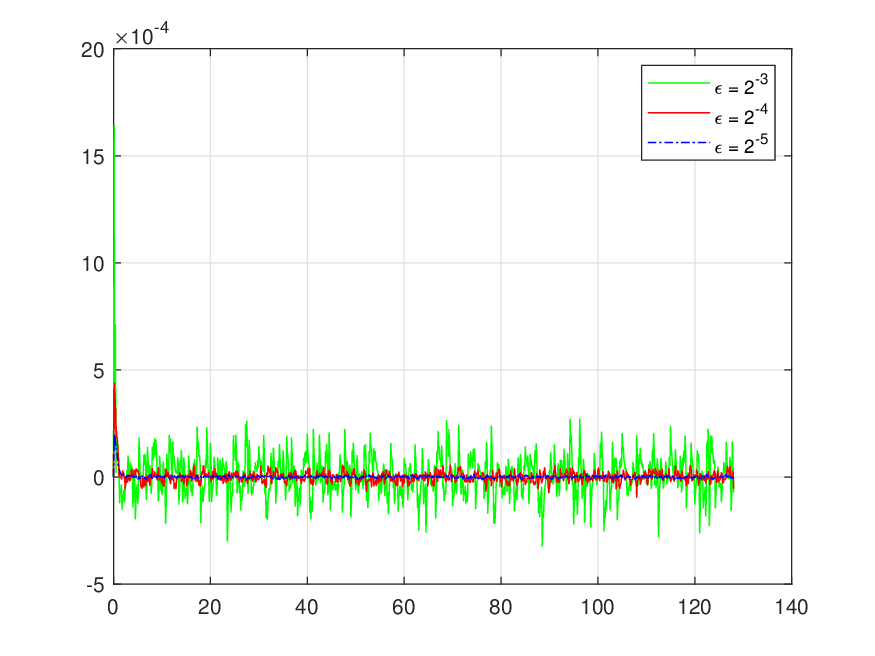}
\caption{For the linear model \eqref{E:example-linear-1}, varying effect of $\BE[X_T^{\eps,\delta}-x_T-\eps Z_T]$ over $300$ sample paths with different values of $\eps$. Here, the error $\BE[X_T^{\eps,\delta}-x_T-\eps Z_T]$ and time are represented on the vertical and horizontal axes, respectively.}\label{F:Error-linear}
\end{center}
\end{figure}

\begin{table}[!h]
\begin{center}
\caption{The absolute maximum and minimum values of $\text{Error} \triangleq \BE[X_T^{\eps,\delta}-x_T-\eps Z_T]$ over $300$ sample paths decrease as the values of $\eps$ decrease.}
\label{Table-linear}
\begin{tabular}{ |p{2cm}|p{2cm}||p{3cm}|p{3cm}|  }
 \hline
 $\eps$ & $\delta=2\eps$ & $|$Maximum Error$|$ & $|$Minimum Error$|$\\
 \hline
 $2^{-3}$ & $2^{-2}$   & $0.00164$    & $3.2175 \times 10^{-4}$ \\
 $2^{-4}$ & $2^{-3}$    &   $4.3641 \times 10^{-4}$  & $9.4164 \times 10^{-5}$  \\
 $2^{-5}$  & $2^{-4}$      &   $1.9649 \times 10^{-4}$  & $1.8798 \times 10^{-5}$  \\
  \hline
\end{tabular}
\end{center}
\end{table}

\subsection*{Example 4}
We consider the non-linear system
\begin{equation}\label{E:example-linear-2}
dX_t^{\eps,\delta}=\left\{f(X_t^{\eps,\delta})+\kappa(X_{\pi_\delta(t)}^{\eps,\delta})\right\}dt+\eps dW_t, \quad X_0^{\eps,\delta}=0.1
\end{equation}
for the functions $f(x) = \frac{\sin x}{1+x^2}-3x,\thinspace \sigma(x)=1$ and $\kappa(x)= \frac{1}{1+e^{-x}}-5x.$ It is easy to check that the mappings $f$ and $\kappa$ satisfy Assumption \ref{A:Poly-Int} as $Df<0$ and $D \kappa<0.$ For the Monte Carlo simulation of the quantity $\sup_{t\ge 0}\left|\BE\varphi(X_t^{\eps,\delta})-\BE\varphi(V_t^{\eps}) \right|^p$ over the $300$ sample paths of Brownian motion, we fix the function $\varphi(x)=x$ and the parameters $T=2^7, X_0^{\eps,\delta}=0.1$ and $\delta=2\eps.$ Figure \ref{F:Error-nonlinear} and Table \ref{Table-nonlinear} show that the absolute error decreases as the values of $\eps$ get small.

\begin{figure}[h!]
\begin{center}
\includegraphics[height=6cm,width=9.6cm]{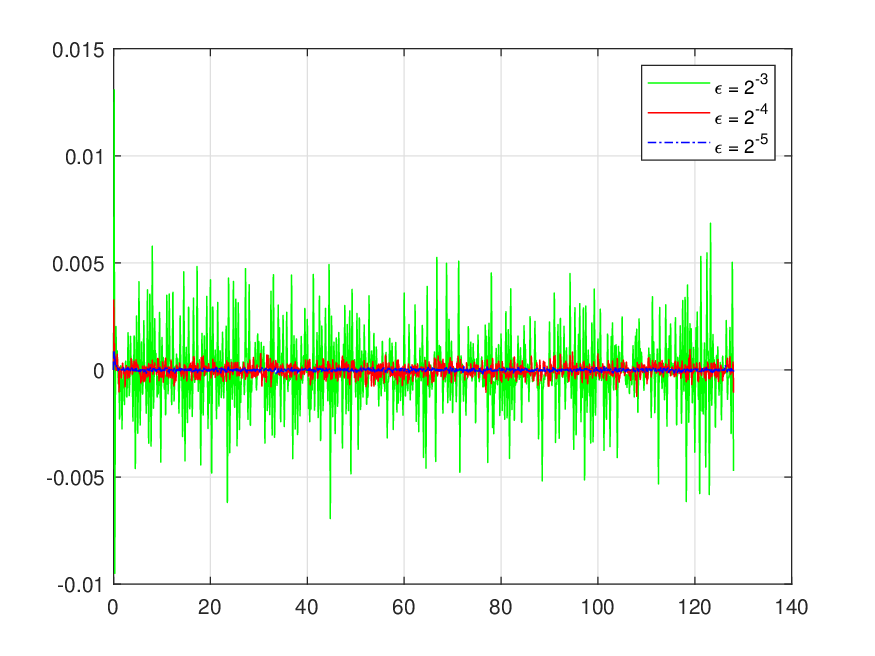}
\caption{For the non-linear model \eqref{E:example-linear-2}, varying effect of $\BE[X_T^{\eps,\delta}-x_T-\eps Z_T]$ over $300$ sample paths with different values of $\eps$. Here, the error $\BE[X_T^{\eps,\delta}-x_T-\eps Z_T]$ and time are represented on the vertical and horizontal axes, respectively.}\label{F:Error-nonlinear}
\end{center}
\end{figure}

\begin{table}[!h]
\begin{center}
\caption{The absolute maximum and minimum values of $\text{Error} \triangleq \BE[X_T^{\eps,\delta}-x_T-\eps Z_T]$ over $300$ sample paths decrease as the values of $\eps$ decrease.}
\label{Table-nonlinear}
\begin{tabular}{|p{2cm}|p{2cm}||p{3cm}|p{3cm}|}
 \hline
 $\eps$ & $\delta=2\eps$ & $|$Maximum Error$|$ & $|$Minimum Error$|$\\
 \hline
 $2^{-3}$ & $2^{-2}$   & $0.0131$    & $0.0095$ \\
 $2^{-4}$ & $2^{-3}$    &   $0.0033$  & $0.0012$  \\
 $2^{-5}$  & $2^{-4}$      &   $9.2848 \times 10^{-4}$  & $2.0363 \times 10^{-4}$  \\

  \hline
\end{tabular}
\end{center}
\end{table}


%


\section{{\sc lln} Type Result: Proof of Theorem \ref{T:LLN-UIT}}\label{S:LLN-Proof}

In this section, we  prove Theorem \ref{T:LLN-UIT} ({\sc lln} Type Result) under Assumptions \ref{A:Poly-Lip} and \ref{A:Poly-growth}. The main building blocks of the proof are Lemmas \ref{L:Poly-p-Moment} and \ref{L:Poly-LLN-Int} below. Here, Lemma \ref{L:Poly-p-Moment} gives the uniform-in-time bound for the quantities $\BE[|X_t^{\eps,\delta}|^p]$ and $x_t$. Lemma \ref{L:Poly-LLN-Int} deals with the quantity $\BE|X_{\pi_\delta(s)}^{\eps,\delta}-X_s^{\eps,\delta}|^p$. In both the results, the generic constant $C$ is always time-independent.

\begin{lemma}\label{L:Poly-p-Moment}
Let $X_t^{\eps,\delta}$ and $x_{t}$ be the solutions of \eqref{E:W-S-SDE-With-pert} and \eqref{E:det-sys} respectively satisfying Assumptions \ref{A:Poly-Lip} and \ref{A:Poly-growth}. Then, for any $p \ge 1$, there exists a positive constant $C_{\ref{L:Poly-p-Moment}}$ such that
\begin{equation*}
\sup_{t \ge 0}\BE\left[|X_t^{\eps,\delta}|^p\right] +\sup_{t \ge 0}|x_{t}|^{p} \le C_{\ref{L:Poly-p-Moment}}.
\end{equation*}
\end{lemma}

\begin{lemma}\label{L:Poly-LLN-Int}
Suppose that $X_t^{\eps,\delta}$ is the solution of \eqref{E:W-S-SDE-With-pert} satisfying Assumptions \ref{A:Poly-Lip} and \ref{A:Poly-growth}. Then, for any $\eps,\delta>0$ and even $p$, there exists a positive constant $C_{\ref{L:Poly-LLN-Int}}$ such that
\begin{equation*}
\sup_{t \ge 0}\BE\left[|X_{\pi_\delta(t)}^{\eps,\delta}-X_t^{\eps,\delta}|^p \right] \le C_{\ref{L:Poly-LLN-Int}}[\delta^p+ \eps^p\delta^{\frac{p}{2}}].
\end{equation*}
\end{lemma}

The proofs of these lemmas are provided in Section \ref{S:SS2-poly}. We now prove Theorem \ref{T:LLN-UIT} under Assumptions \ref{A:Poly-Lip} and \ref{A:Poly-growth}.
\begin{proof}[Proof of Theorem \ref{T:LLN-UIT}]
We start by noting that it suffices to prove the statement for even powers $p$. For $p=1$, the required bound follows from H\"{o}lder's inequality. So, let us take $p$ to be an even integer. Using the integral representation for the process $X_t^{\eps,\delta}$ and $x_t$, we get
\begin{equation*}
X_t^{\eps,\delta}-x_t = \int_0^t [{f(X_s^{\eps,\delta})-f(x_s)}] ds +\int_0^t {[\kappa(X_{\pi_\delta(s)}^{\eps,\delta})-\kappa(x_s)]} ds + \eps\int_0^t \sigma(X_s^{\eps,\delta})dW_s.
\end{equation*}
Applying It\^{o}'s formula to the function $\varphi(t,x)=e^{\frac{tp \lambda}{2}}|x|^p,$  we have
\begin{equation*}
\begin{aligned}
e^{\frac{tp \lambda}{2}}|X_t^{\eps,\delta}-x_t|^p
&= \int_0^t \frac{ p \lambda}{2} e^{\frac{s p \lambda}{2}}(X_s^{\eps,\delta}-x_s)^p \thinspace ds + p \int_0^t e^{\frac{s p \lambda}{2}}(X_s^{\eps,\delta}-x_s)^{p-1}\eps  \sigma(X_s^{\eps,\delta})\thinspace dW_s\\
& \qquad \qquad \qquad \qquad  + \int_0^t p e^{\frac{s p \lambda}{2}}(X_s^{\eps,\delta}-x_s)^{p-1}\left[{f(X_s^{\eps,\delta})-f(x_s)}\right] ds \\
& \qquad \qquad \qquad \qquad \qquad \quad + \int_0^t p e^{\frac{s p \lambda}{2}}(X_s^{\eps,\delta}-x_s)^{p-1}\left[\kappa(X_{\pi_\delta(s)}^{\eps,\delta})-\kappa(x_s)\right] ds \\
& \qquad \qquad \qquad \qquad \qquad \qquad \qquad \quad + \frac{p(p-1)}{2}\int_0^t e^{\frac{s p \lambda}{2}}(X_s^{\eps,\delta}-x_s)^{p-2}\eps^2 \sigma^2(X_s^{\eps,\delta})\thinspace ds.
\end{aligned}
\end{equation*}
Next, using Assumption \ref{A:Poly-Lip} (in particular, Equation \eqref{E:Poly-Diss-Lip}), we obtain
\begin{equation*}
\begin{aligned}
e^{\frac{tp \lambda}{2}}|X_t^{\eps,\delta}-x_t|^p &\le  \int_0^t \frac{ p \lambda}{2} e^{\frac{s p \lambda}{2}}(X_s^{\eps,\delta}-x_s)^p \thinspace ds + p \int_0^t e^{\frac{s p \lambda}{2}}(X_s^{\eps,\delta}-x_s)^{p-1}\eps  \sigma(X_s^{\eps,\delta})\thinspace dW_s\\
& \qquad  -\lambda \int_0^t p e^{\frac{s p \lambda}{2}}(X_s^{\eps,\delta}-x_s)^{p}\thinspace ds +\int_0^t p e^{\frac{s p \lambda}{2}}(X_s^{\eps,\delta}-x_s)^{p-1}\left[\kappa(X_{\pi_\delta(s)}^{\eps,\delta})-\kappa(x_s)\right] ds \\
& \qquad \qquad \qquad \qquad \qquad \qquad \qquad \qquad + \frac{p(p-1)}{2}\int_0^t e^{\frac{s p \lambda}{2}}(X_s^{\eps,\delta}-x_s)^{p-2}\eps^2 \sigma^2(X_s^{\eps,\delta})\thinspace ds.
\end{aligned}
\end{equation*}
%
We further take expectation (the stochastic integral is zero due to Lemma \ref{L:Poly-p-Moment}) and use Young's inequality to the term $(X_s^{\eps,\delta}-x_s)^{p-2}\eps^2$ (with the H\"{o}lder conjugates $\frac{p}{p-2},\thinspace \frac{p}{2}$) to get a sufficiently small $\xi>0$ and boundedness of $\sigma$ to obtain
\begin{equation*}
\begin{aligned}
\BE \left[e^{\frac{tp \lambda}{2}}|X_t^{\eps,\delta}-x_t|^p \right] & \le \left(\frac{ p \lambda}{2}-\lambda p+\xi \frac{p(p-1)}{2}\frac{(p-2)}{p}\|\sigma^{2}\|_{\infty} \right) \int_0^t e^{\frac{s p \lambda}{2}}\BE (X_s^{\eps,\delta}-x_s)^p \thinspace ds\\
& + \int_0^t p e^{\frac{s p \lambda}{2}}\BE\left[(X_s^{\eps,\delta}-x_s)^{p-1} [\kappa(X_{\pi_\delta(s)}^{\eps,\delta})-\kappa(x_s)]\right]ds + \eps^p \frac{p(p-1)}{p} \frac{\|\sigma^{2}\|_{\infty}}{\xi^{\frac{p-2}{p}}} \int_0^t e^{\frac{s p \lambda}{2}} ds.
\end{aligned}
\end{equation*}
Now, in the above expression, first writing $\kappa(X_{\pi_\delta(s)}^{\eps,\delta})-\kappa(x_s)$ as $\kappa(X_{\pi_\delta(s)}^{\eps,\delta})-\kappa(X_s^{\eps,\delta})+\kappa(X_s^{\eps,\delta})-\kappa(x_s)$, then using $(X_s^{\eps,\delta}-x_s)^{p-1} [\kappa(X_{\pi_\delta(s)}^{\eps,\delta})-\kappa(x_s)] \le |X_s^{\eps,\delta}-x_s|^{p-1} |\kappa(X_{\pi_\delta(s)}^{\eps,\delta})-\kappa(x_s)|$ followed by Lipschitz continuity of $\kappa$, we obtain
\begin{equation*}
\begin{aligned}
\BE \left[e^{\frac{tp \lambda}{2}}|X_t^{\eps,\delta}-x_t|^p \right] & \le   p\left(-\left(\frac{  \lambda}{2}-L_{\kappa}\right) + \xi \frac{(p-1)(p-2)}{2p}\|\sigma^{2}\|_{\infty}\right)\int_0^t e^{\frac{s p \lambda}{2}}\BE (X_s^{\eps,\delta}-x_s)^p \thinspace ds\\
&  + L_{\kappa}\int_0^t p e^{\frac{s p \lambda}{2}}\BE\left[|X_s^{\eps,\delta}-x_s|^{p-1} |X_{\pi_\delta(s)}^{\eps,\delta}-X_s^{\eps,\delta}|\right]ds + \frac{\eps^p (p-1)  \|\sigma^{2}\|_{\infty}}{\xi^{\frac{p-2}{p}}} \int_0^t e^{\frac{s p \lambda}{2}}\thinspace ds.
\end{aligned}
\end{equation*}
Using the generalized Young's inequality $ab\leq \frac{\eta}{r_1}a^{r_1}+\frac{1}{\eta^{r_2/r_1}r_2}b^{r_2}$ with the H\"{o}lder conjugates $r_1=\frac{p}{p-1},\thinspace r_2=p$, any $\eta>0$ and Lemma \ref{L:Poly-LLN-Int}, we obtain for $\eta>0$ and $\xi>0$ to be chosen
\begin{align*}
\BE \left[e^{\frac{tp \lambda}{2}}|X_t^{\eps,\delta}-x_t|^p \right] & \le   -p\left(\frac{\lambda}{2}-\left(1+\eta\frac{p-1}{p}\right)L_{\kappa} -\xi \frac{(p-1)(p-2)}{2p}\|\sigma^{2}\|_{\infty}\right) \int_0^t e^{\frac{s p \lambda}{2}}\BE [|X_s^{\eps,\delta}-x_s|^p] \thinspace ds\\
& \qquad \qquad \qquad \qquad \qquad \qquad \qquad \qquad \qquad \qquad \qquad + C [\eps^p+\delta^p+ \eps^p\delta^{\frac{p}{2}}] \int_0^t e^{\frac{s p \lambda}{2}}\thinspace ds,
\end{align*}
for some constant $C<\infty$ that depends on $p,\eta,\xi,L_{\kappa},\|\sigma^{2}\|_{\infty}$. Given that we have assumed $\frac{\lambda}{2}-L_{\kappa}>0$, we can choose $\eta,\xi>0$ both sufficiently small such that  one has
\begin{align*}
\frac{  \lambda}{2}-\left(1+\eta\frac{p-1}{p}\right)L_{\kappa} -\xi \frac{(p-1)(p-2)}{2p}\|\sigma^{2}\|_{\infty}&>0.
\end{align*}
This then yields
\begin{align*}
\BE \left[e^{\frac{tp \lambda}{2}}|X_t^{\eps,\delta}-x_t|^p \right] & \le    C [\eps^p+\delta^p+ \eps^p\delta^{\frac{p}{2}}] \int_0^t e^{\frac{s p \lambda}{2}}\thinspace ds,
\end{align*}
yielding the uniform-in-time bound
\begin{align*}
\BE \left[|X_t^{\eps,\delta}-x_t|^p \right] & \le    C [\eps^p+\delta^p+ \eps^p\delta^{\frac{p}{2}}] e^{-\frac{tp \lambda}{2}} \int_0^t e^{\frac{s p \lambda}{2}}\thinspace ds\leq  C [\eps^p+\delta^p+ \eps^p\delta^{\frac{p}{2}}],
\end{align*}
with a potentially different constant $C<\infty$ that is independent of time $t\in\BR_{+}$. This completes the proof of the theorem.
\end{proof}

\subsection{Proof of Lemmas \ref{L:Poly-p-Moment} and \ref{L:Poly-LLN-Int}}\label{S:SS2-poly}

\begin{proof}[Proof of Lemma \ref{L:Poly-p-Moment}]
We will only prove the bound for $\sup_{t \ge 0}\BE[|X_t^{\eps,\delta}|^p]$ as the bound for $\sup_{t\ge 0}|x_{t}|$ (and consequently for $\sup_{t\ge 0}|x_{t}|^{p}$) follows by the exact same argument and is simpler.

We note that it is just sufficient to  prove the statement for even powers $p$.  We start by using It\^{o}'s formula to the function $\varphi(t,x)=e^{\frac{tp \alpha}{2}}|x|^p,$ and noting $e^{\frac{tp \alpha}{2}}|x|^p=e^{\frac{tp \alpha}{2}}x^p$ to yield
\begin{equation*}
\begin{aligned}
e^{\frac{tp \alpha}{2}}|X_t^{\eps,\delta}|^p &= |X_0^{\eps,\delta}|^p+ \int_0^t \frac{ p \alpha}{2} e^{\frac{s p \alpha}{2}}(X_s^{\eps,\delta})^p \thinspace ds + \int_0^t p e^{\frac{s p \alpha}{2}}(X_s^{\eps,\delta})^{p-1}\left[f(X_s^{\eps,\delta})+\kappa(X_{\pi_\delta(s)}^{\eps,\delta})\right] ds \\
& \qquad \qquad \quad  +\frac{p(p-1)}{2}\int_0^t e^{\frac{s p \alpha}{2}}(X_s^{\eps,\delta})^{p-2}\eps^2 \sigma^2(X_s^{\eps,\delta})\thinspace ds + p \int_0^t e^{\frac{s p \alpha}{2}}(X_s^{\eps,\delta})^{p-1}\eps  \sigma(X_s^{\eps,\delta})\thinspace dW_s.
\end{aligned}
\end{equation*}
Therefore,
\begin{equation*}
\begin{aligned}
e^{\frac{tp \alpha}{2}}|X_t^{\eps,\delta}|^p & \le  |X_0^{\eps,\delta}|^p+\int_0^t pe^{\frac{s p \alpha}{2}} \left[\frac{\alpha}{2} |X_s^{\eps,\delta}|^2 + X_s^{\eps,\delta} \cdot f(X_s^{\eps,\delta}) \right] |X_s^{\eps,\delta}|^{p-2} ds \\
& \qquad \qquad \qquad + \int_0^t p e^{\frac{s p \alpha}{2}}(X_s^{\eps,\delta})^{p-1}|\kappa(X_{\pi_\delta(s)}^{\eps,\delta}|\thinspace ds
+  \frac{p(p-1)}{2}\int_0^t e^{\frac{s p \alpha}{2}}|X_s^{\eps,\delta}|^{p-2}\eps^2 |\sigma(X_s^{\eps,\delta})|^2\thinspace ds \\
& \qquad \qquad \qquad \qquad \qquad \qquad \qquad \qquad \qquad \qquad \qquad \quad + p \int_0^t e^{\frac{s p \alpha}{2}}(X_s^{\eps,\delta})^{p-1}\eps  \sigma(X_s^{\eps,\delta})\thinspace dW_s.
\end{aligned}
\end{equation*}
Next, using (\ref{E:Poly-diss-f}) and inequality $(X_t^{\eps,\delta})^{p-1}\le |X_t^{\eps,\delta}|^{p-1} $ followed by taking expectation, we get for some time-independent constant $C<\infty$
\begin{multline*}
e^{\frac{tp \alpha}{2}} \BE [|X_t^{\eps,\delta}|^p] \le \BE |X_0^{\eps,\delta}|^p +\int_0^t -\frac{p \alpha}{2}  \BE \left[e^{\frac{s p \alpha}{2}}|X_s^{\eps,\delta}|^p \right] ds + \int_0^t p \|\kappa\|_{\infty} \BE \left[e^{\frac{s p \alpha}{2}}|X_s^{\eps,\delta}|^{p-1} \right] ds \\
 + \left(p\beta+\eps^{2}\frac{p(p-1)\|\sigma^{2}\|_{\infty}}{2}\right)\int_0^t \BE \left[e^{\frac{s p \alpha}{2}}|X_s^{\eps,\delta}|^{p-2}\right] ds.
\end{multline*}
We further employ the generalized Young's inequality to the terms $(p\|\kappa\|_{\infty})|X_s^{\eps,\delta}|^{p-1}$ (with the H\"{o}lder conjugates $\frac{p}{p-1},\thinspace p$) and $\left(p\beta+\eps^{2}\frac{p(p-1)\|\sigma^{2}\|_{\infty}}{2}\right)|X_s^{\eps,\delta}|^{p-2}$ (with the H\"{o}lder conjugates $\frac{p}{p-2},\thinspace \frac{p}{2}$) obtaining for some positive constants $\eta_1, \thinspace \eta_2$ to be chosen that
\begin{align*}
e^{\frac{tp \alpha}{2}} \BE [|X_t^{\eps,\delta}|^p]& \le  -\left(\frac{p\alpha}{2}-\eta_1 (p-1)\|\kappa\|_{\infty}-\eta_2 \frac{p-2}{p}\left(p\beta+\eps^{2}\frac{p(p-1)\|\sigma^{2}\|_{\infty}}{2}\right) \right) \int_0^t   \BE \left[e^{\frac{s p \alpha}{2}}|X_s^{\eps,\delta}|^{p} \right] ds  \nonumber\\
&\quad \qquad \qquad \qquad \qquad \qquad \qquad \qquad \qquad \qquad \qquad \qquad \qquad \quad   +\BE |X_0^{\eps,\delta}|^p  + C \int_0^t e^{\frac{s p \alpha}{2}}\thinspace ds .
\end{align*}
Choosing now $\eta_1,\eta_2>0$ sufficiently small so that
\begin{align*}
\left(\frac{p\alpha}{2}-\eta_1 (p-1)\|\kappa\|_{\infty}-\eta_2 \frac{p-2}{p}\left(p\beta+\eps^{2}\frac{p(p-1)\|\sigma^{2}\|_{\infty}}{2}\right) \right)&>0
\end{align*}
yields the bound
 $\BE [|X_t^{\eps,\delta}|^p] \le e^{-\frac{tp \alpha}{2}}|x_{0}|^p + C e^{-\frac{tp \alpha}{2}} \int_0^t e^{\frac{s p \alpha}{2}}\thinspace ds <C$ 
for some time-independent constant $C<\infty$ and for all $\eps,\delta>0$. This completes the proof of lemma.
\end{proof}

\begin{proof}[Proof of Lemma \ref{L:Poly-LLN-Int}]
From the integral representation of $X_t^{\eps,\delta},$ we get
$$\left|{X_{\pi_\delta(s)}^{\varepsilon, \delta}-X_s^{\varepsilon, \delta}}\right|^p \le C \left|\int_{\pi_\delta(s)}^s\left\{f(X_u^{\eps,\delta})+\kappa(X_{\pi_\delta(u)}^{\eps,\delta})\right\}du \right|^p +  C \left|\int_{\pi_\delta(s)}^s \eps \sigma(X_u^{\eps,\delta})dW_u \right|^p.$$
For some $q \in \BN$, the boundedness of $\kappa$, Assumption \ref{A:Poly-Lip} and the martingale moment inequalities \cite[Proposition 3.26]{KS91} yield
\begin{equation*}
\begin{aligned}
\BE\left[|{X_{\pi_\delta(s)}^{\varepsilon, \delta}-X_s^{\varepsilon, \delta}}|^p \right] & \le C \int_{[{\pi_\delta(s)},s]^p}\sum_{i=1}^p\left\{1+\BE |X_{r_i}^{\eps,\delta}|^{q}\right\} dr_1\cdots dr_p  + C \eps^p \BE \left\{\int_{[{\pi_\delta(s)},s]} \sigma^2(X_{u}^{\eps,\delta})\thinspace d{u}\right\}^{\frac{p}{2}}ds.
\end{aligned}
\end{equation*}
Finally, using boundedness of $\sigma$ from Assumption \ref{A:Poly-growth}, Lemma \ref{L:Poly-p-Moment} and the fact $s-\pi_\delta(s)< \delta$, we obtain the required bound.
\end{proof}

\section{{\sc clt} Type Result: Proof of Theorem \ref{T:fluctuations-R-1-2}}\label{S:CLT-Proof}

We now prove Theorem \ref{T:fluctuations-R-1-2} 
 with Assumptions \ref{A:Poly-Lip}, \ref{A:Poly-growth}, \ref{A:Poly-Int}.
 An organization of this section is as follows. In Section \ref{S:CLT-Prop}, we prove Theorem \ref{T:fluctuations-R-1-2} through its main building blocks Propositions \ref{P:Central-Prop}, \ref{P:I2-Prop} and  \ref{P:I3-Prop-Rem}. The proofs of Propositions \ref{P:Central-Prop} through
\ref{P:I3-Prop-Rem} are provided in Section \ref{S:CLT-Proof-Prop}. Proposition \ref{P:Central-Prop} plays a key role in the proof of Theorem
\ref{T:fluctuations-R-1-2} and is proved through a series of helpful lemmas. The proofs of these lemmas are deferred to Section \ref{S:CLT-Lemmas-Proof}.

\subsection{Proof of Theorem \ref{T:fluctuations-R-1-2}}\label{S:CLT-Prop}


Before proving Theorem \ref{T:fluctuations-R-1-2}, let us get more insights into the fluctuation process $Z^{\eps,\delta}_t $ and its limiting process $Z_t$ solving \eqref{E:Limiting-SDE-Fluctuations}. Recalling the integral representation for the process $X_t^{\eps,\delta}$ and $x_t$ from equations \eqref{E:W-S-SDE-With-pert} and \eqref{E:det-sys}, respectively, followed by Taylor's theorem, we have

\begin{multline}\label{E:Fluctuation-R-1-2}
\begin{aligned}
Z_t^{\eps,\delta} &= \int_0^t \frac{f(X_s^{\eps,\delta})-f(x_s)}{\eps} ds +\int_0^t \frac{[\kappa(X_{\pi_\delta(s)}^{\eps,\delta})-\kappa(X_s^{\eps,\delta})]}{\eps} ds
 + \int_0^t \frac{[\kappa(X_s^{\eps,\delta})-\kappa(x_s)]}{\eps}ds \\
 & \qquad \qquad \qquad \qquad \qquad \qquad \qquad \qquad \qquad \quad \quad \qquad \qquad \qquad \qquad \qquad \qquad  + \int_0^t \sigma(X_s^{\eps,\delta})dW_s\\
 & = \int_0^t [Df(x_s)+D\kappa (x_s)]Z_s^{\eps,\delta}\thinspace ds - \int_0^t D\kappa (x_s)\frac{X_s^{\varepsilon, \delta}-X_{\pi_\delta(s)}^{\varepsilon, \delta}}{\eps}\thinspace ds
 + \int_0^t \sigma(X_s^{\eps,\delta})dW_s+ \int_0^t {\sf R}_s^{\eps,\delta}\thinspace ds,
 \end{aligned}
\end{multline}
where,
\begin{equation}\label{E:Remainder-terms}
\begin{aligned}
{\sf R}_s^{\eps,\delta} \triangleq \sum_{i=1}^3{\sf R}_i^{\eps,\delta}(s) &\triangleq \left[\frac{f(X_s^{\eps,\delta})-f(x_s)}{\eps}- Df(x_s)Z_s^{\eps,\delta} \right]
+\left[\frac{\kappa(X_s^{\eps,\delta})-\kappa(x_s)}{\eps}-D\kappa(x_s)Z_s^{\eps,\delta}\right]\\
& \qquad \qquad \qquad \qquad \qquad \quad+ \left[\frac{\kappa(X_{\pi_\delta(s)}^{\eps,\delta})-\kappa(X_s^{\eps,\delta})}{\eps}-D\kappa(x_s)\frac{X_{\pi_\delta(s)}^{\varepsilon, \delta}-X_s^{\varepsilon, \delta}}{\eps} \right].
\end{aligned}
\end{equation}
One can now identify the limiting {\sc sde} (that is, equation \eqref{E:Limiting-SDE-Fluctuations}) for the process $Z_t$ from equation \eqref{E:Fluctuation-R-1-2} as follows. This identification is essentially obtained by replacing $Z_t^{\eps,\delta}$ by $Z_t$, $\int_0^t D\kappa (x_s)\frac{X_s^{\varepsilon, \delta}-X_{\pi_\delta(s)}^{\varepsilon, \delta}}{\eps}\thinspace ds$ by $\frac{\cc}{2} \int_0^t D \kappa(x_s)[f(x_s)+\kappa(x_s)] \thinspace ds$, $X_t^{\eps,\delta}$ by $x_t$ in the stochastic integral $\int_0^t \sigma(X_s^{\eps,\delta})dW_s$, and finally showing the remainder term $\int_0^t {\sf R}_s^{\eps,\delta}\thinspace ds$ is small. These approximations are made precise in Propositions \ref{P:Central-Prop}, \ref{P:I2-Prop} and \ref{P:I3-Prop-Rem} stated below.


In Regimes $1$ and $2$, we see that rescaled process $Z_t^{\eps,\delta}\triangleq \frac{X_t^{\eps,\delta}-x_t}{\eps}$ and the effective process $Z_t$ solve equations \eqref{E:Fluctuation-R-1-2} and \eqref{E:Limiting-SDE-Fluctuations}, respectively. To show the convergence of $Z_t^{\eps,\delta}$ to $Z_t$, as $\eps,\delta \searrow 0,$ one first need to demonstrate the following convergence in an appropriate sense.
\begin{multline}\label{E:ell}
\lim_{\substack{\eps,\delta \searrow 0\\ \delta/\eps \to \cc}} \int_0^t e^{\int_s^t[Df(x_u)+D\kappa (x_u)]du}D\kappa(x_s)\frac{X_s^{\varepsilon, \delta}-X_{\pi_\delta(s)}^{\varepsilon, \delta}}{\eps}\thinspace ds \\ =\frac{\cc}{2} \int_0^t e^{\int_s^t[Df(x_u)+D\kappa (x_u)]du}D\kappa(x_s)[f(x_s)+\kappa(x_s)] \thinspace ds.
\end{multline}
In the fluctuation study, equation \eqref{E:ell} enables us to identify the extra drift term capturing both the fast sampling and small noise effects. In Proposition \ref{P:Central-Prop} below, we show the convergence of  $\int_0^t e^{\int_s^t[Df(x_u)+D\kappa (x_u)]du}D\kappa(x_s)\frac{X_s^{\varepsilon, \delta}-X_{\pi_\delta(s)}^{\varepsilon, \delta}}{\eps}\thinspace ds$ to the term $\frac{\cc}{2} \int_0^t e^{\int_s^t[Df(x_u)+D\kappa (x_u)]du}D\kappa(x_s)[f(x_s)+\kappa(x_s)] \thinspace ds$, as $\eps,\delta \searrow 0.$ In Propositions \ref{P:Central-Prop} through \ref{P:I3-Prop-Rem} below, the constant $C$ is always time-independent.

\begin{proposition}\label{P:Central-Prop}
Let $X_{t}^{\eps,\delta}$ be the solution of {\sc sde} \eqref{E:W-S-SDE-With-pert} and
$$I_1^{\eps,\delta}(t)\triangleq \left|\int_0^t e^{\int_s^t[Df(x_u)+D\kappa (x_u)]du}\left[\frac{\cc}{2}  D \kappa(x_s)[f(x_s)+\kappa(x_s)]- D\kappa (x_s)\frac{X_s^{\varepsilon, \delta}-X_{\pi_\delta(s)}^{\varepsilon, \delta}}{\eps}\right]ds\right|^p,$$
where, $p \ge 1$ is an integer. Then, for all sufficiently small $\eps,\delta>0,$ there exists a positive constant $C_{\ref{P:Central-Prop}}$ such that
$$\sup_{ t \ge 0}\BE \left[I_1^{\eps,\delta}(t)\right] \le C_{\ref{P:Central-Prop}}\left[\frac{\delta^{2p}}{\eps^p}+\left|\frac{\delta}{\eps}-\cc \right|^p+\left(\frac{\delta^{2p}+\delta^p}{\eps^p}\right)(\delta^p+\eps^p+{\delta}^{\frac{p}{2}}\eps^p)+\delta^{p/2} \right].$$
\end{proposition}

In Propositions \ref{P:I2-Prop} and \ref{P:I3-Prop-Rem} stated below, we show that the terms $I_2^{\eps,\delta}(t)$ and $I_3^{\eps,\delta}(t)$ are small in an appropriate sense as $\eps,\delta$ vanish.
\begin{proposition}\label{P:I2-Prop}
Let $X_{t}^{\eps,\delta}$ be the solution of {\sc sde} \eqref{E:W-S-SDE-With-pert} and
$$I_2^{\eps,\delta}(t) \triangleq \left|\int_0^t e^{\int_s^t[Df(x_u)+D\kappa (x_u)]du}\left[\sigma(X_s^{\eps,\delta})-\sigma(x_s)\right]dW_s\right|^p,$$
where, $p \ge 1$ is an integer. Then, for all sufficiently small $\eps,\delta>0,$ there exists a positive constant $C_{\ref{P:I2-Prop}}$ such that
 $$\sup_{ t \ge 0}\BE \left[I_2^{\eps,\delta}(t)\right] \le C_{\ref{P:I2-Prop}} \left(\delta^{2p}+\eps^{2p}+{\delta}^p\eps^{2p} \right)^{\frac{1}{2}}.$$
\end{proposition}

\begin{proposition}\label{P:I3-Prop-Rem}
Let $X_{t}^{\eps,\delta}$ be the solution of {\sc sde} \eqref{E:W-S-SDE-With-pert}, ${\sf R}_{t}^{\eps,\delta}$ be defined as in equation \eqref{E:Remainder-terms}, and
\begin{equation}\label{E:I3}
I_3^{\eps,\delta}(t)\triangleq \left|\int_0^t e^{\int_s^t[Df(x_u)+D\kappa (x_u)]du} {\sf R}_s^{\eps,\delta}\thinspace ds\right|^p,
\end{equation}
where, $p \ge 1$ is an integer. Then, for all sufficiently small $\eps,\delta>0,$ there exists a positive constant $C_{\ref{P:I3-Prop-Rem}}$ such that
 $$\sup_{ t \ge 0}\BE \left[I_3^{\eps,\delta}(t)\right] \le \frac{C_{\ref{P:I3-Prop-Rem}}}{\eps^p}(\eps^{2p}+\delta^{2p}+\eps^{2p} \delta^p).$$
\end{proposition}

We are now ready to prove Theorem \ref{T:fluctuations-R-1-2}.
\begin{proof}[Proof of Theorem \ref{T:fluctuations-R-1-2}]
Using the integral representations for the stochastic processes $Z_t^{\eps,\delta}$ and $Z_t$ given by equations \eqref{E:Fluctuation-R-1-2} and \eqref{E:Limiting-SDE-Fluctuations}, respectively, we have
\begin{equation*}
\begin{aligned}
Z_t^{\eps,\delta}-Z_t&= \int_0^t [Df(x_s)+D\kappa (x_s)](Z_s^{\eps,\delta}-Z_s)\thinspace ds + \int_0^t \left[\sigma(X_s^{\eps,\delta})-\sigma(x_s)\right]dW_s+ \int_0^t {\sf R}_s^{\eps,\delta}\thinspace ds\\
&\qquad \qquad \qquad \qquad \qquad \quad + \left[\frac{\cc}{2} \int_0^t D \kappa(x_s)[f(x_s)+\kappa(x_s)] \thinspace ds- \int_0^t D\kappa (x_s)\frac{X_s^{\varepsilon, \delta}-X_{\pi_\delta(s)}^{\varepsilon, \delta}}{\eps}\thinspace ds \right].
\end{aligned}
\end{equation*}
We now employ It{\^o}'s formula to the function $\varphi(t,x)=e^{-\int_0^t[Df(x_u)+D\kappa (x_u)]du}x$ to get
\begin{equation*}
\begin{aligned}
Z_t^{\eps,\delta}-Z_t&= \int_0^t e^{\int_s^t[Df(x_u)+D\kappa (x_u)]du}\left[\frac{\cc}{2}  D \kappa(x_s)[f(x_s)+\kappa(x_s)]- D\kappa (x_s)\frac{X_s^{\varepsilon, \delta}-X_{\pi_\delta(s)}^{\varepsilon, \delta}}{\eps}\right]ds \\
& \qquad \qquad + \int_0^t e^{\int_s^t[Df(x_u)+D\kappa (x_u)]du}\left[\sigma(X_s^{\eps,\delta})-\sigma(x_s)\right]dW_s+ \int_0^t e^{\int_s^t[Df(x_u)+D\kappa (x_u)]du} {\sf R}_s^{\eps,\delta}\thinspace ds.
\end{aligned}
\end{equation*}
Hence, for any integer $p\ge 1,$
\begin{equation}\label{E:Is}
\begin{aligned}
|Z_t^{\eps,\delta}-Z_t|^p&\le C \left|\int_0^t e^{\int_s^t[Df(x_u)+D\kappa (x_u)]du}\left[\frac{\cc}{2}D \kappa(x_s)[f(x_s)+\kappa(x_s)]- D\kappa (x_s)\frac{X_s^{\varepsilon, \delta}-X_{\pi_\delta(s)}^{\varepsilon, \delta}}{\eps}\right]ds\right|^p \\
&+C \left|\int_0^t e^{\int_s^t[Df(x_u)+D\kappa (x_u)]du}\left[\sigma(X_s^{\eps,\delta})-\sigma(x_s)\right]dW_s\right|^p+ C\left|\int_0^t e^{\int_s^t[Df(x_u)+D\kappa (x_u)]du} {\sf R}_s^{\eps,\delta}\thinspace ds\right|^p\\
& \triangleq C \left [I_1^{\eps,\delta}(t)+I_2^{\eps,\delta}(t)+I_3^{\eps,\delta}(t)\right].
\end{aligned}
\end{equation}
Finally, putting Propositions \ref{P:Central-Prop}, \ref{P:I2-Prop} and \ref{P:I3-Prop-Rem} together in equation \eqref{E:Is}, we get
\begin{equation*}\label{E:CLT-Fatou}
\sup_{ t \ge 0}\BE\left[|Z_t^{\eps,\delta}-Z_t|^p \right] \le \\ C\left[\frac{\delta^{2p}}{\eps^p}+\left|\frac{\delta}{\eps}-\cc \right|^p+\left(\frac{\delta^p}{\eps^p}+1 \right)(\delta^p+\eps^p+{\delta}^{\frac{p}{2}}\eps^p)^{\frac{1}{2}}+\delta^{p/2}+\eps^p\right].
\end{equation*}
\end{proof}

\subsection{Proof of Propositions \ref{P:Central-Prop} through \ref{P:I3-Prop-Rem}}\label{S:CLT-Proof-Prop}
In this section, we provide the proofs of Propositions \ref{P:Central-Prop} through \ref{P:I3-Prop-Rem}. To prove Proposition \ref{P:Central-Prop}, we first state (without proof) a series of auxiliary Lemmas \ref{L:M-terms} to \ref{L:M4-estimate} below. The proofs of the aforementioned lemmas are postponed to Section \ref{S:CLT-Lemmas-Proof}.

Moving in the direction of the proof of Proposition \ref{P:Central-Prop}, Lemma \ref{L:M-terms} gives the decomposition of the term $(1/\eps) D\kappa(x_s) \left[{X_s^{\varepsilon, \delta}-X_{\pi_\delta(s)}^{\varepsilon, \delta}}\right]$ which appears in the dynamics of fluctuations given by \eqref{E:Fluctuation-R-1-2}.
\begin{lemma}\label{L:M-terms}
Let $X_{t}^{\eps,\delta}$ be the solution of {\sc sde} \eqref{E:W-S-SDE-With-pert}. Then,
 $$(1/\eps) D\kappa(x_s) \left[{X_s^{\varepsilon, \delta}-X_{\pi_\delta(s)}^{\varepsilon, \delta}}\right]= \sum_{i=1}^{4} {\sfM}_i^{\eps,\delta}(s),\quad \text{where}$$
\begin{equation}\label{E:M1234}
\begin{aligned}
{\sf M}_1^{\eps,\delta}(s) & \triangleq (1/\eps) D\kappa (x_s) f(x_{\pi_\delta(s)})[{s-\pi_\delta(s)}], \qquad
{\sf M}_2^{\eps,\delta}(s) \triangleq (1/\eps) D\kappa (x_s)\kappa(X_{\pi_\delta(r)}^{\eps,\delta})[{s-\pi_\delta(s)}], \\
{\sf M}_3^{\eps,\delta}(s) & \triangleq \frac{1}{\eps} D\kappa (x_s)\int_{\pi_\delta(s)}^{s}\{{f(X_r^{\eps,\delta})-f(X_{\pi_\delta(r)}^{\eps,\delta})}\} dr  \\
 & \qquad \qquad \qquad \qquad \qquad \qquad \quad \qquad \quad+ \frac{1}{\eps} D\kappa (x_s) [{f(X_{\pi_\delta(s)}^{\varepsilon, \delta})-f(x_{\pi_\delta(s)})}][s-\pi_\delta(s)],\\
{\sf M}_4^{\eps,\delta}(s) & \triangleq  D\kappa (x_s) \int_{\pi_\delta(s)}^{s}\sigma(X_u^{\eps,\delta})dW_u.
\end{aligned}
\end{equation}
\end{lemma}

Now, Lemma \ref{L:M1-estimate} stated below deals with the deterministic term ${\sf M}_1^{\eps,\delta}(t)$ and shows that ${\sf M}_1^{\eps,\delta}(t)$ is close, uniform-in-time, to the first part of the effective drift term (i.e, $\frac{\cc}{2}D\kappa(x_t)f(x_t)$, see \eqref{E:ell}).
\begin{lemma}\label{L:M1-estimate}
Let ${\sf M}_1^{\eps,\delta}(t)$ be defined as in equation \eqref{E:M1234} and $p\ge 1$ be an integer. Then, there exists a positive constant $C_{\ref{L:M1-estimate}}$ such that
\begin{equation*}
\sup_{t \ge 0}\left|\int_0^t e^{\int_s^t[Df(x_u)+D\kappa (x_u)]du}\left[{\sf M}_1^{\eps,\delta}(s)-\frac{\cc}{2}D\kappa(x_s)f(x_s) \right] ds \right|^p\le C_{\ref{L:M1-estimate}} \left|\frac{\delta}{\eps}-\cc \right|^p + \frac{\delta^{2p}}{\eps^p} C_{\ref{L:M1-estimate}}.
\end{equation*}
\end{lemma}
Next, Lemma \ref{L:M2-estimate} handles the asymptotic analysis, uniform-in-time, of the process ${\sf M}_2^{\eps,\delta}(t)$ as $\eps,\delta$ tend to zero and shows that ${\sf M}_2^{\eps,\delta}(t)$ is close to the second part of the effective drift term (i.e, $\frac{\cc}{2}D\kappa(x_t)\kappa(x_t)$, see equation \eqref{E:ell}).

\begin{lemma}\label{L:M2-estimate}
Let ${\sf M}_2^{\eps,\delta}(t)$ be defined as in equation \eqref{E:M1234} and $p\ge 1$ be an integer. Then, for all sufficiently small $\eps,\delta>0,$ there exists a positive constant $C_{\ref{L:M2-estimate}}$ such that
\begin{multline*}
\sup_{t \ge 0}\BE\left|\int_0^t e^{\int_s^t[Df(x_u)+D\kappa (x_u)]du}\left[{\sf M}_2^{\eps,\delta}(s)-\frac{\cc}{2}D\kappa(x_s)\kappa(x_s) \right] ds \right|^p\\
\le C_{\ref{L:M2-estimate}}\left[\frac{\delta^{2p}}{\eps^p}+\left|\frac{\delta}{\eps}-\cc \right|^p+\left(\frac{\delta^{2p}+\delta^p}{\eps^p}\right)(\delta^p+\eps^p+{\delta}^{\frac{p}{2}}\eps^p) \right].
\end{multline*}
\end{lemma}
Finally, Lemmas \ref{L:M3-estimate} and \ref{L:M4-estimate} deal with the processes ${\sf M}_3^{\eps,\delta}(t)$ and ${\sf M}_4^{\eps,\delta}(t)$, respectively and show that the terms of interest are small in uniform time setting.

\begin{lemma}\label{L:M3-estimate}
Let ${\sf M}_3^{\eps,\delta}(t)$ be defined as in equation \eqref{E:M1234} and $p\ge 1$ be an integer. Then, for all sufficiently small $\eps,\delta>0,$ there exists a positive constant $C_{\ref{L:M3-estimate}}$ such that
\begin{equation*}
\sup_{t \ge 0}\BE\left| \int_0^t e^{\int_s^t[Df(x_u)+D\kappa (x_u)]du}{\sf M}_3^{\eps,\delta}(s) \thinspace ds \right|^p \le C_{\ref{L:M3-estimate}}\frac{\delta^p}{\eps^p}(\delta^p+\eps^p+{\delta}^{\frac{p}{2}}\eps^p)+C_{\ref{L:M3-estimate}}\frac{\delta^{2p}}{\eps^p}.
\end{equation*}
\end{lemma}

\begin{lemma}\label{L:M4-estimate}
Let ${\sf M}_4^{\eps,\delta}(t)$ be defined as in equation \eqref{E:M1234} and $p\ge 1$ be an integer. Then, there exists a positive constant $C_{\ref{L:M4-estimate}}$ such that
\begin{equation*}
\sup_{t \ge 0}\BE\left| \int_0^t e^{\int_s^t[Df(x_u)+D\kappa (x_u)]du}{\sf M}_4^{\eps,\delta}(s) \thinspace ds \right|^p \le C_{\ref{L:M4-estimate}}\delta^{\frac{p}{2}}.
\end{equation*}
\end{lemma}

We now prove Proposition \ref{P:Central-Prop}.
\begin{proof}[Proof of Proposition \ref{P:Central-Prop}]
For any integer $p\ge 1,$ using Lemma \ref{L:M-terms} for the decomposition of the process  $D\kappa (x_s)\frac{X_s^{\varepsilon, \delta}-X_{\pi_\delta(s)}^{\varepsilon, \delta}}{\eps}$, we have
\begin{equation*}
\begin{aligned}
\sup_{t \ge 0}\BE \left[I_1^{\eps,\delta}(t)\right] & \le C\sup_{t\ge 0}\left|\int_0^t e^{\int_s^t[Df(x_u)+D\kappa (x_u)]du}\left[{\sf M}_1^{\eps,\delta}(s)-\frac{\cc}{2}D\kappa(x_s)f(x_s) \right] ds \right|^p\\
& \qquad \quad  + C \sup_{t \ge 0} \BE\left|\int_0^t e^{\int_s^t[Df(x_u)+D\kappa (x_u)]du}\left[{\sf M}_2^{\eps,\delta}(s)-\frac{\cc}{2}D\kappa(x_s)\kappa(x_s) \right] ds \right|^p\\
& \qquad \qquad \qquad \qquad \qquad \qquad \qquad + C \sup_{t \ge 0}\BE \left| \int_0^t e^{\int_s^t[Df(x_u)+D\kappa (x_u)]du}{\sf M}_3^{\eps,\delta}(s) \thinspace ds \right|^p\\
& \qquad \qquad \qquad \qquad \qquad \qquad \qquad \qquad+ C \sup_{t \ge 0}\BE \left| \int_0^t e^{\int_s^t[Df(x_u)+D\kappa (x_u)]du}{\sf M}_4^{\eps,\delta}(s) \thinspace ds \right|^p.
\end{aligned}
\end{equation*}
Now, employing Lemmas \ref{L:M1-estimate}, \ref{L:M2-estimate}, \ref{L:M3-estimate} and \ref{L:M4-estimate} for the terms on the right hand side of the above equation, we get the required bound.
\end{proof}

Next, we present the proof of Proposition \ref{P:I2-Prop}.

\begin{proof}[Proof of Proposition \ref{P:I2-Prop}]
Using the Burkholder-Davis-Gundy inequalities followed by Jensen's inequality\footnote{For a random variable $Y$ and a concave function $\psi$ with $\BE|Y|, \BE|\psi(Y)|<\infty$, we have $\BE[\psi(Y)]\le \psi(\BE[Y])$.} for concave functions, we have
\begin{equation*}
\begin{aligned}
\BE \left[I_2^{\eps,\delta}(t)\right] &\le C  \BE \left[\left(\int_0^t e^{2\int_s^t[Df(x_u)+D\kappa (x_u)]du}|\sigma(X_s^{\eps,\delta})-\sigma(x_s)|^2   ds\right)^{\frac{p}{2}}\right]\\
& \qquad \qquad \qquad \qquad \qquad \qquad  \le C  \left[\BE\left(\int_0^t e^{2\int_s^t[Df(x_u)+D\kappa (x_u)]du}|\sigma(X_s^{\eps,\delta})-\sigma(x_s)|^2   ds\right)^{p}\right]^{\frac{1}{2}}.
\end{aligned}
\end{equation*}
Now, employing the Lipschitz continuity of $\sigma,$ and Assumption \ref{A:Poly-Int},
we obtain from the above equation
\begin{equation*}
\begin{aligned}
\BE \left[I_2^{\eps,\delta}(t)\right]&\le C  \left[\BE \left(\int_0^t e^{2\int_s^t[Df(x_u)+D\kappa (x_u)]du}|X_s^{\eps,\delta}-x_s|^2ds\right)^p\right]^{\frac{1}{2}}\\
& \qquad = C \left[\BE \int_{[0,t]^p}\left( \prod_{i=1}^p e^{2\int_s^t[Df(x_u)+D\kappa (x_u)]du}|X_{s_i}^{\eps,\delta}-x_{s_i}|^2\right)ds_1\cdots ds_p \right]^{\frac{1}{2}}\\
& \qquad \qquad \le C \left[\BE \int_{[0,t]^p}\left( \prod_{i=1}^p e^{2\int_s^t[Df(x_u)+D\kappa (x_u)]du}\right)\left(\sum_{i=1}^p|X_{s_i}^{\eps,\delta}-x_{s_i}|^{2p}\right)ds_1\cdots ds_p \right]^{\frac{1}{2}}\\
& \qquad \qquad \qquad \quad \le C \left[ \int_{[0,t]^p}\left( \prod_{i=1}^p e^{2\int_s^t[Df(x_u)+D\kappa (x_u)]du}\right)\left(\sum_{i=1}^p \BE |X_{s_i}^{\eps,\delta}-x_{s_i}|^{2p}\right)ds_1\cdots ds_p \right]^{\frac{1}{2}}.
\end{aligned}
\end{equation*}
After using a simple algebra and Theorem \ref{T:LLN-UIT}, we obtain the desired result.
\end{proof}

Finally, we prove Proposition \ref{P:I3-Prop-Rem}.

\begin{proof}[Proof of Proposition \ref{P:I3-Prop-Rem}]
For any integer $p\ge 1,$ recalling the definition of the process $I_3^{\eps,\delta}(t)$ from equation \eqref{E:I3}, we have
\begin{equation*}
\sup_{t \ge 0}\BE \left[I_3^{\eps,\delta}(t)\right] \le \sup_{t \ge 0}\BE \left[{\sf R}_1^{\eps,\delta}(t)\right] +\sup_{t \ge 0}\BE \left[{\sf R}_2^{\eps,\delta}(t)\right] +\sup_{t \ge 0}\BE \left[{\sf R}_3^{\eps,\delta}(t)\right],
\end{equation*}
where ${\sf R}_i^{\eps,\delta}(s),\thinspace i=1,2,3$ are defined in equation \eqref{E:Remainder-terms}. Using Proposition \ref{P:CLT-R-12} for the terms \\ $\sup_{t \ge 0}\BE \left[{\sf R}_1^{\eps,\delta}(t)\right]$ and $\sup_{t \ge 0}\BE \left[{\sf R}_2^{\eps,\delta}(t)\right]$ and Proposition \ref{P:CLT-R-3} for the term $\sup_{t \ge 0}\BE \left[{\sf R}_3^{\eps,\delta}(t)\right]$, we obtain the required bound.
\end{proof}

\subsection{Proof of Lemmas \ref{L:M-terms} through \ref{L:M4-estimate}}\label{S:CLT-Lemmas-Proof}
\begin{proof}[Proof of Lemma \ref{L:M-terms}]
For any $s \ge 0,$ exploiting the integral representation for $X_s^{\eps,\delta}$ from \eqref{E:W-S-SDE-With-pert} followed by a simple algebra, we get
\begin{equation*}
\begin{aligned}
\frac{X_s^{\eps,\delta}-X^{\eps,\delta}_{\pi_\delta(s)}}{\eps}&=\int_{\pi_\delta(s)}^{s}\frac{f(X_r^{\eps,\delta})+\kappa(X_{\pi_\delta(r)}^{\eps,\delta})}{\eps}\thinspace dr + \int_{\pi_\delta(s)}^{s}\sigma(X_u^{\eps,\delta})dW_u\\
&= \int_{\pi_\delta(s)}^{s} \frac{f(X_r^{\eps,\delta})-f(X_{\pi_\delta(r)}^{\eps,\delta})}{\eps} \thinspace dr + \int_{\pi_\delta(s)}^{s} \frac{f(X_{\pi_\delta(r)}^{\eps,\delta})+\kappa(X_{\pi_\delta(r)}^{\eps,\delta})}{\eps}\thinspace dr \\
& \qquad \qquad  \qquad \qquad \qquad \qquad \qquad \qquad \qquad \qquad \qquad \qquad \quad+ \int_{\pi_\delta(s)}^{s}\sigma(X_u^{\eps,\delta})dW_u.
\end{aligned}
\end{equation*}
Writing $f(X_{\pi_\delta(r)}^{\eps,\delta})+\kappa(X_{\pi_\delta(r)}^{\eps,\delta})$ in the middle term of the above equation as $f(X_{\pi_\delta(r)}^{\eps,\delta})-f(x_{\pi_\delta(s)})+\kappa(X_{\pi_\delta(r)}^{\eps,\delta})+f(x_{\pi_\delta(s)})$, we see that
\begin{multline*}
D\kappa(x_s)\frac{X_s^{\varepsilon, \delta}-X_{\pi_\delta(s)}^{\varepsilon, \delta}}{\eps} = D\kappa(x_s) f(x_{\pi_\delta(s)})\frac{s-\pi_\delta(s)}{\eps} + D\kappa(x_s)\int_{\pi_\delta(s)}^{s}\frac{\kappa(X_{\pi_\delta(r)}^{\eps,\delta})}{\eps}\thinspace dr \\
\qquad  + D\kappa(x_s)\int_{\pi_\delta(s)}^{s}\frac{f(X_r^{\eps,\delta})-f(X_{\pi_\delta(r)}^{\eps,\delta})}{\eps} \thinspace dr
 + D\kappa(x_s) \frac{f(X_{\pi_\delta(s)}^{\varepsilon, \delta})-f(x_{\pi_\delta(s)})}{\eps}(s-\pi_\delta(s))\\
 + D\kappa(x_s)\int_{\pi_\delta(s)}^{s}\sigma(X_u^{\eps,\delta})dW_u.
\end{multline*}
Defining right hand side of the above equation as $\sum_{i=1}^{4} {\sfM}_i^{\eps,\delta}(t)$, we finish the proof of the lemma.
\end{proof}

\begin{proof}[Proof of Lemma \ref{L:M1-estimate}]
We recall the definition of ${\sf M}_1^{\eps,\delta}(s)$ from equation \eqref{E:M1234}.
For any $s \ge 0,$ writing ${\sf M}_1^{\eps,\delta}(s)$ as ${\sf M}_1^{\eps,\delta}(s)+(1/\eps) D\kappa (x_{\pi_\delta(s)}) f(x_{\pi_\delta(s)})[{s-\pi_\delta(s)}]-(1/\eps) D\kappa (x_{\pi_\delta(s)}) f(x_{\pi_\delta(s)})[{s-\pi_\delta(s)}]$, we get

\begin{equation*}
\begin{aligned}
{\sf M}_1^{\eps,\delta}(s)-\frac{\cc}{2}D\kappa(x_s)f(x_s)& =(1/\eps) \{D\kappa (x_s)-D\kappa (x_{\pi_\delta(s)})\} f(x_{\pi_\delta(s)})[{s-\pi_\delta(s)}]\\
&\qquad \qquad\qquad\qquad+(1/\eps) D\kappa (x_{\pi_\delta(s)}) f(x_{\pi_\delta(s)})[{s-\pi_\delta(s)}]-\frac{\cc}{2}D\kappa(x_s)f(x_s)\\
& =(1/\eps) \{D\kappa (x_s)-D\kappa (x_{\pi_\delta(s)})\} f(x_{\pi_\delta(s)})[{s-\pi_\delta(s)}]\\
&\qquad \qquad\qquad\qquad+ (1/\eps) D\kappa (x_{\pi_\delta(s)}) f(x_{\pi_\delta(s)})[{s-\pi_\delta(s)}]-\frac{\delta}{2 \eps}D\kappa(x_s)f(x_s)\\
&\qquad \qquad\qquad\qquad \qquad \qquad \qquad \qquad \qquad \qquad+ \frac{1}{2}\left(\frac{\delta}{\eps}-\cc \right)D\kappa(x_s)f(x_s).
\end{aligned}
\end{equation*}
Now, multiplying both sides of the above equation by $e^{\int_s^t[Df(x_u)+D\kappa (x_u)]du},$ we have
\begin{equation*}
\begin{aligned}
e^{\int_s^t[Df(x_u)+D\kappa (x_u)]du}&\left[{\sf M}_1^{\eps,\delta}(s) -\frac{\cc}{2}D\kappa(x_s)f(x_s) \right] \\
&  = (1/\eps) \{D\kappa (x_s)-D\kappa (x_{\pi_\delta(s)})\} f(x_{\pi_\delta(s)})[{s-\pi_\delta(s)}]e^{\int_s^t[Df(x_u)+D\kappa (x_u)]du}\\
& \quad + \left[(1/\eps) D\kappa (x_{\pi_\delta(s)}) f(x_{\pi_\delta(s)})[{s-\pi_\delta(s)}]-\frac{\delta}{2 \eps}D\kappa(x_s)f(x_s)\right]e^{\int_s^t[Df(x_u)+D\kappa (x_u)]du}\\
&\qquad \qquad    \qquad \qquad \qquad  \qquad \quad+ \frac{1}{2}\left(\frac{\delta}{\eps}-\cc \right)D\kappa(x_s)f(x_s)e^{\int_s^t[Df(x_u)+D\kappa (x_u)]du}.
\end{aligned}
\end{equation*}
Next, for any integer $p\ge 1,$ we use Lemma \ref{L:Sec-der-Kappa} to handle the term $D\kappa (x_s)-D\kappa (x_{\pi_\delta(s)})$, Assumption \ref{A:Poly-Int} and the hypothesis that $f$ is of polynomial growth and $\sup_{t\ge 0}|x_t| < \infty$ (Lemma \ref{L:Poly-p-Moment}) to get
\begin{equation*}
\begin{aligned}
 & \left| \int_0^t e^{\int_s^t[Df(x_u)+D\kappa (x_u)]du}\left[{\sf M}_1^{\eps,\delta}(s) -\frac{\cc}{2}D\kappa(x_s)f(x_s) \right] ds \right|^p \\
&\le \frac{\delta^{2p}}{\eps^p}C\left(\int_0^t e^{\int_s^t[Df(x_u)+D\kappa (x_u)]du}\thinspace ds\right)^p + \frac{1}{2^p}C\left|\frac{\delta}{\eps}-\cc \right|^p\left(\int_0^t e^{\int_s^t[Df(x_u)+D\kappa (x_u)]du}\thinspace ds\right)^p\\
&\qquad \qquad \quad + C\left| \int_0^t \left[(1/\eps) D\kappa (x_{\pi_\delta(s)}) f(x_{\pi_\delta(s)})[{s-\pi_\delta(s)}]-\frac{\delta}{2 \eps}D\kappa(x_s)f(x_s)\right]e^{\int_s^t[Df(x_u)+D\kappa (x_u)]du}\thinspace ds\right|^p\\
& \le C \left(\frac{\delta^{2p}}{\eps^p}+\left|\frac{\delta}{\eps}-\cc \right|^p \right)\\
&\qquad \qquad \quad + C\left| \int_0^t \left[(1/\eps) D\kappa (x_{\pi_\delta(s)}) f(x_{\pi_\delta(s)})[{s-\pi_\delta(s)}]-\frac{\delta}{2 \eps}D\kappa(x_s)f(x_s)\right]e^{\int_s^t[Df(x_u)+D\kappa (x_u)]du}\thinspace ds\right|^p.
\end{aligned}
\end{equation*}
The proof can be concluded by employing Lemma \ref{L:M-1-S-1} for the last term on the right side.
\end{proof}

\begin{proof}[Proof of Lemma \ref{L:M2-estimate}]
Recalling the definition of the process ${\sf M}_2^{\eps,\delta}(s)$ from equation \eqref{E:M1234} and using a simple algebra, we obtain
\begin{equation*}
\begin{aligned}
{\sf M}_2^{\eps,\delta}(s)-\frac{\cc}{2}D\kappa(x_s)\kappa(x_s)&  =(1/\eps) \{D\kappa (x_s)-D\kappa (x_{\pi_\delta(s)})\} \kappa(x_{\pi_\delta(s)})[{s-\pi_\delta(s)}]\\
& \qquad+ (1/\eps) \{D\kappa (x_s)-D\kappa (x_{\pi_\delta(s)})\} \left[\kappa(X_{\pi_\delta(s)}^{\eps,\delta})-\kappa (x_{\pi_\delta(s)})\right][{s-\pi_\delta(s)}]\\
& \qquad \qquad + (1/\eps) D\kappa (x_{\pi_\delta(s)}) \left[\kappa(X_{\pi_\delta(s)}^{\eps,\delta})-\kappa (x_{\pi_\delta(s)})\right][{s-\pi_\delta(s)}]\\
& \qquad \qquad \qquad+ (1/\eps) D\kappa (x_{\pi_\delta(s)}) \kappa(x_{\pi_\delta(s)})[{s-\pi_\delta(s)}]-\frac{\delta}{2 \eps}D\kappa(x_s)\kappa(x_s)\\
& \qquad \qquad \qquad \qquad \qquad \qquad \qquad \qquad \qquad \qquad + \frac{1}{2}\left(\frac{\delta}{\eps}-\cc \right)D\kappa(x_s)\kappa(x_s).
\end{aligned}
\end{equation*}
Now, multiplying both sides of the above equation by $e^{\int_s^t[Df(x_u)+D\kappa (x_u)]du},$ we have
\begin{equation*}
\begin{aligned}
e^{\int_s^t[Df(x_u)+D\kappa (x_u)]du}&\left[ {\sf M}_2^{\eps,\delta}(s)-\frac{\cc}{2}D\kappa(x_s)\kappa(x_s) \right]\\
&= (1/\eps) \{D\kappa (x_s)-D\kappa (x_{\pi_\delta(s)})\} \kappa(x_{\pi_\delta(s)})[{s-\pi_\delta(s)}]e^{\int_s^t[Df(x_u)+D\kappa (x_u)]du}\\
&+ \frac{1}{\eps} \{D\kappa (x_s)-D\kappa (x_{\pi_\delta(s)})\} \left[\kappa(X_{\pi_\delta(s)}^{\eps,\delta})-\kappa (x_{\pi_\delta(s)})\right][{s-\pi_\delta(s)}]e^{\int_s^t[Df(x_u)+D\kappa (x_u)]du}\\
&+ (1/\eps) D\kappa (x_{\pi_\delta(s)}) \left[\kappa(X_{\pi_\delta(s)}^{\eps,\delta})-\kappa (x_{\pi_\delta(s)})\right][{s-\pi_\delta(s)}]e^{\int_s^t[Df(x_u)+D\kappa (x_u)]du}\\
&+ \left[(1/\eps) D\kappa (x_{\pi_\delta(s)}) \kappa(x_{\pi_\delta(s)})[{s-\pi_\delta(s)}]-\frac{\delta}{2 \eps}D\kappa(x_s)\kappa(x_s)\right]e^{\int_s^t[Df(x_u)+D\kappa (x_u)]du}\\
&+ \frac{1}{2}\left(\frac{\delta}{\eps}-\cc \right)D\kappa(x_s)\kappa(x_s)e^{\int_s^t[Df(x_u)+D\kappa (x_u)]du}.
\end{aligned}
\end{equation*}
Using the Lipschitz continuity and boundedness of the function $\kappa$ and its derivatives, the fact $s-\pi_\delta(s)<\delta,$ for any $\delta>0$ and $\sup_{t\ge 0}|x_t| < \infty$ (Lemma \ref{L:Poly-p-Moment}), we have
\begin{equation*}
\begin{aligned}
& \left|\int_0^t e^{\int_s^t[Df(x_u)+D\kappa (x_u)]du}\left[{\sf M}_2^{\eps,\delta}(s)-\frac{\cc}{2}D\kappa(x_s)\kappa(x_s) \right]ds \right|^p \\
& \qquad \qquad \qquad \le C\frac{\delta^p}{\eps^p}\left[\int_0^t|D\kappa (x_s)-D\kappa (x_{\pi_\delta(s)})|e^{\int_s^t[Df(x_u)+D\kappa (x_u)]du}\thinspace ds \right]^p  \\
&\qquad \qquad  \qquad   + C\left|\int_0^t \left[\frac{1}{\eps} D\kappa (x_{\pi_\delta(s)}) \kappa(x_{\pi_\delta(s)})[{s-\pi_\delta(s)}]-\frac{\delta}{2 \eps}D\kappa(x_s)\kappa(x_s)\right]e^{\int_s^t[Df(x_u)+D\kappa (x_u)]du}\thinspace ds\right|^p\\
&\qquad \qquad \qquad + C\frac{\delta^p}{\eps^p}\left[ \int_0^t |D\kappa (x_s)-D\kappa (x_{\pi_\delta(s)})| \left|X_{\pi_\delta(s)}^{\eps,\delta}-x_{\pi_\delta(s)}\right|e^{\int_s^t[Df(x_u)+D\kappa (x_u)]du}\thinspace ds\right]^p
\end{aligned}
\end{equation*}
\begin{equation*}
\begin{aligned}
&\qquad \qquad \qquad + C\frac{\delta^p}{\eps^p}\left[ \int_0^t \left|X_{\pi_\delta(s)}^{\eps,\delta}-x_{\pi_\delta(s)}\right|e^{\int_s^t[Df(x_u)+D\kappa (x_u)]du}\thinspace ds\right]^p \\
&\qquad \qquad \qquad+ C\left|\frac{\delta}{\eps}-\cc \right|^p\left[\int_0^t e^{\int_s^t[Df(x_u)+D\kappa (x_u)]du}\thinspace ds\right]^p.
\end{aligned}
\end{equation*}
Taking expectation on both sides, we have
\begin{equation*}
\begin{aligned}
& \BE\left[\left|\int_0^t e^{\int_s^t[Df(x_u)+D\kappa (x_u)]du}\left[{\sf M}_2^{\eps,\delta}(s)-\frac{\cc}{2}D\kappa(x_s)\kappa(x_s) \right]ds \right|^p\right] \\
& \qquad \qquad \le C\frac{\delta^p}{\eps^p}\left[\int_0^t|D\kappa (x_s)-D\kappa (x_{\pi_\delta(s)})|e^{\int_s^t[Df(x_u)+D\kappa (x_u)]du}\thinspace ds \right]^p  \\
&\qquad \qquad \quad + C\left|\int_0^t \left[(1/\eps) D\kappa (x_{\pi_\delta(s)}) \kappa(x_{\pi_\delta(s)})[{s-\pi_\delta(s)}]-\frac{\delta}{2 \eps}D\kappa(x_s)\kappa(x_s)\right]e^{\int_s^t[Df(x_u)+D\kappa (x_u)]du}\thinspace ds\right|^p\\
&\qquad \qquad \qquad \quad + C\frac{\delta^p}{\eps^p}\BE\left[ \int_0^t |D\kappa (x_s)-D\kappa (x_{\pi_\delta(s)})| \left|X_{\pi_\delta(s)}^{\eps,\delta}-x_{\pi_\delta(s)}\right|e^{\int_s^t[Df(x_u)+D\kappa (x_u)]du}\thinspace ds\right]^p\\
&\qquad \qquad \qquad \qquad \qquad+ C\frac{\delta^p}{\eps^p}\BE\left[\int_0^t \left|X_{\pi_\delta(s)}^{\eps,\delta}-x_{\pi_\delta(s)}\right|e^{\int_s^t[Df(x_u)+D\kappa (x_u)]du}\thinspace ds\right]^p\\
& \qquad \qquad \qquad \qquad \qquad \qquad \qquad \qquad+ C\left|\frac{\delta}{\eps}-\cc \right|^p\left[\int_0^t e^{\int_s^t[Df(x_u)+D\kappa (x_u)]du}\thinspace ds\right]^p.
\end{aligned}
\end{equation*}
Employing Lemma \ref{L:Sec-der-Kappa} to handle the term $\sup_{s \ge 0}|D\kappa (x_u)-D\kappa (x_{\pi_\delta(u)})|,$ we obtain
\begin{equation*}
\begin{aligned}
& \BE\left[\left|\int_0^t e^{\int_s^t[Df(x_u)+D\kappa (x_u)]du}\left[{\sf M}_2^{\eps,\delta}(s)-\frac{\cc}{2}D\kappa(x_s)\kappa(x_s) \right]ds \right|^p\right] \\
& \qquad \qquad  \le C\frac{\delta^{2p}}{\eps^p}\left[\int_0^t e^{\int_s^t[Df(x_u)+D\kappa (x_u)]du}\thinspace ds \right]^p + C\left|\frac{\delta}{\eps}-\cc \right|^p\left[\int_0^t e^{\int_s^t[Df(x_u)+D\kappa (x_u)]du}\thinspace ds\right]^p  \\
&\qquad \qquad \quad   + C\left|\int_0^t \left[(1/\eps) D\kappa (x_{\pi_\delta(s)}) \kappa(x_{\pi_\delta(s)})[{s-\pi_\delta(s)}]-\frac{\delta}{2 \eps}D\kappa(x_s)\kappa(x_s)\right]e^{\int_s^t[Df(x_u)+D\kappa (x_u)]du}\thinspace ds\right|^p\\
&\qquad \qquad \qquad \qquad \qquad \qquad \qquad \quad+ C\left(\frac{\delta^{2p}+\delta^p}{\eps^p}\right)\BE\left[ \int_0^t  \left|X_{\pi_\delta(s)}^{\eps,\delta}-x_{\pi_\delta(s)}\right|e^{\int_s^t[Df(x_u)+D\kappa (x_u)]du}\thinspace ds\right]^p.
\end{aligned}
\end{equation*}
Finally, we use Assumption \ref{A:Poly-Int} to handle the term $\int_0^t e^{\int_s^t[Df(x_u)+D\kappa (x_u)]du}\thinspace ds$ and  Lemmas \ref{L:M-2-S-2} and \ref{L:M-2-Intermediate} for the last two terms in the above equation to complete the proof.
\end{proof}

\begin{proof}[Proof of Lemma \ref{L:M3-estimate}]
Recalling the definition of the process ${\sf M}_3^{\eps,\delta}(s)$ from  equation \eqref{E:M1234}, we have

\begin{equation*}
\begin{aligned}
 \int_0^t e^{\int_s^t[Df(x_u)+D\kappa (x_u)]du}&{\sf M}_3^{\eps,\delta}(s) \thinspace ds\\
&  = \int_0^t\frac{1}{\eps} e^{\int_s^t[Df(x_u)+D\kappa (x_u)]du}D\kappa (x_s)\left(\int_{\pi_\delta(s)}^{s}\{{ f(X_r^{\eps,\delta})- f(X_{\pi_\delta(r)}^{\eps,\delta})}\}\thinspace dr \right)ds  \\
 & \qquad+ \int_0^t\frac{1}{\eps}e^{\int_s^t[Df(x_u)+D\kappa (x_u)]du} D\kappa (x_s) \left[{ f(X_{\pi_\delta(s)}^{\varepsilon, \delta})-f(x_{\pi_\delta(s)})}\right][s-\pi_\delta(s)]\thinspace ds.
\end{aligned}
\end{equation*}
Now, for any $p\ge 1,$ using the boundedness of the first-order derivatives of $\kappa$, the fact $s-\pi_\delta(s)< \delta,$ for any $\delta>0,$ and a simple algebra, we have after taking expectation
\begin{equation*}
\begin{aligned}
\BE\left|\int_0^t e^{\int_s^t[Df(x_u)+D\kappa (x_u)]du}{\sf M}_3^{\eps,\delta}(s) \thinspace ds\right|^p &\le \frac{C}{\eps^p}\BE\left[\int_0^t e^{\int_s^t[Df(x_u)+D\kappa (x_u)]du}\int_{\pi_\delta(s)}^{s}\left|{ f(X_r^{\eps,\delta})- f(X_{\pi_\delta(r)}^{\eps,\delta})}\right| dr \thinspace ds \right]^p\\
& \quad \quad \quad + C\frac{\delta^p}{\eps^p}\BE\left[\int_0^t e^{\int_s^t[Df(x_u)+D\kappa (x_u)]du}\left|{ f(X_{\pi_\delta(s)}^{\varepsilon, \delta})-f(x_{\pi_\delta(s)})}\right| ds \right]^p\\
& \le  \frac{C}{\eps^p}\BE\left[\int_0^t e^{\int_s^t[Df(x_u)+D\kappa (x_u)]du}\int_{\pi_\delta(s)}^{s} |f(X_r^{\eps,\delta})- f(x_r)|dr\thinspace ds \right]^p\\
& \thinspace  + \frac{C}{\eps^p}\BE\left[\int_0^t e^{\int_s^t[Df(x_u)+D\kappa (x_u)]du}\int_{\pi_\delta(s)}^{s} |f(X_{\pi_\delta(r)}^{\eps,\delta})-f(x_{\pi_\delta(r)})|dr\thinspace ds \right]^p\\
& \quad    +  \frac{C}{\eps^p}\BE\left[\int_0^t e^{\int_s^t[Df(x_u)+D\kappa (x_u)]du}\int_{\pi_\delta(s)}^{s}\left|f(x_r)- f(x_{\pi_\delta(r)})\right|dr \thinspace ds \right]^p\\
& \qquad    +  C\frac{\delta^p}{\eps^p}\BE\left[\int_0^t e^{\int_s^t[Df(x_u)+D\kappa (x_u)]du}\left|{ f(X_{\pi_\delta(s)}^{\varepsilon, \delta})-f(x_{\pi_\delta(s)})}\right| ds \right]^p.
\end{aligned}
\end{equation*}

The next step is to use the local Lipschitz continuity of the mapping $f$ by Assumption \ref{A:Poly-Lip} together with H\"{o}lder's inequality, Theorem \ref{T:LLN-UIT}
and Lemmas \ref{L:Poly-p-Moment} and \ref{L:Sampling-Difference}. We omit the details for brevity.

\end{proof}

\begin{proof}[Proof of Lemma \ref{L:M4-estimate}]
Let us define for notational convenience
\begin{equation*}
J^{\eps,\delta}(t) \triangleq \left|\int_0^t e^{\int_s^t[Df(x_u)+D\kappa (x_u)]du}D\kappa (x_s)\int_{\pi_\delta(s)}^s \sigma(X_u^{\eps,\delta})\thinspace dW_u\thinspace ds \right|^p.
\end{equation*}

Using the basic integral inequality for Lebesgue integrals $|\int \cdot|\le \int|\cdot|,$ we have
\begin{equation*}
\begin{aligned}
J^{\eps,\delta}(t) &\le C \left(\int_0^t e^{\int_s^t[Df(x_u)+D\kappa (x_u)]du}\left|\int_{\pi_\delta(s)}^s \sigma(X_u^{\eps,\delta})\thinspace dW_u\right| ds \right)^p\\
& \quad  = C  \left[\int_{[0,t]^p} \left(\prod_{i=1}^p e^{\int_{s_i}^t[Df(x_u)+D\kappa (x_u)]du}\right) \times\left(\prod_{i=1}^p\left|\int_{[{\pi_\delta(s_i)},s_i]} \sigma(X_{u_i}^{\eps,\delta})\thinspace dW_{u_i}\right| \right)ds_1\cdots ds_p \right]\\
& \qquad \le C \left[\int_{[0,t]^p} \left(\prod_{i=1}^p e^{\int_{s_i}^t[Df(x_u)+D\kappa (x_u)]du}\right) \times\left(\sum_{i=1}^p\left|\int_{[{\pi_\delta(s_i)},s_i]} \sigma(X_{u_i}^{\eps,\delta})\thinspace dW_{u_i}\right|^p \right)ds_1\cdots ds_p \right].
\end{aligned}
\end{equation*}
Taking expectation followed by martingale moment inequalities \cite[Proposition 3.26]{KS91}, we obtain
\begin{multline*}
\BE\left[J^{\eps,\delta}(t)\right] \le C  \left[\int_{[0,t]^p}\left( \prod_{i=1}^p e^{\int_{s_i}^t[Df(x_u)+D\kappa (x_u)]du}\right)\left(\sum_{i=1}^p\BE \left\{\int_{[{\pi_\delta(s_i)},s_i]} \sigma^2(X_{u_i}^{\eps,\delta})\thinspace d{u_i}\right\}^{\frac{p}{2}} \right)ds_1\cdots ds_p \right].
\end{multline*}
Employing the boundedness of $\sigma$, Assumption \ref{A:Poly-growth}, the fact $s-\pi_\delta(s)\le \delta,$ and a simple algebra, we obtain the desired estimate.

\end{proof}

\section{Proof of more supporting results: Lemmas \ref{L:Sampling-Difference} to \ref{L:M-2-Intermediate} and Propositions \ref{P:CLT-R-12},\ref{P:CLT-R-3}}\label{S:Error-Prop-proof}
\subsection{Proof of Lemmas \ref{L:Sampling-Difference} through \ref{L:M-2-Intermediate}}

\begin{lemma}\label{L:Sampling-Difference}
Let $x_{t}$ be the solution of \eqref{E:det-sys}. Then, for any integer $p \ge 1,$ there exists a positive constant $C_{\ref{L:Sampling-Difference}}$ such that
$$\sup_{t \ge 0}|x_t - x_{\pi_\delta(t)}|^p \le \delta^pC_{\ref{L:Sampling-Difference}}.$$
\end{lemma}
\begin{proof}
Recalling the integral representation of $x_t$ from equation \eqref{E:det-sys} and using the polynomial growth of the mappings $f$ and boundedness of $\kappa$, we have for some $q<\infty$
\begin{equation*}
|x_t - x_{\pi_\delta(t)}| \le C \int_{\pi_\delta(t)}^{t} \left(1+ \sup_{ u \ge 0}|x_u|^{q} \right)ds.
\end{equation*}
Using the fact $t-\pi_\delta(t)<\delta$ for any $\delta>0$ and  $\sup_{ u \ge 0}|x_u|< \infty$ (Lemma \ref{L:Poly-p-Moment}), we get the required estimate.
\end{proof}

\begin{lemma}\label{L:Sec-der-Kappa}
Let $x_{t}$ be the solution of equation \eqref{E:det-sys}. Then, for any $t\ge 0,$ there exists a positive constant $C_{\ref{L:Sec-der-Kappa}}$ such that
\begin{equation*}
\sup_{t \ge 0}\left|D\kappa (x_t)-D\kappa (x_{\pi_\delta(t)})\right|\le \delta C_{\ref{L:Sec-der-Kappa}}.
\end{equation*}
\end{lemma}

\begin{proof}
Using Taylor's theorem, we have
\begin{equation*}
\left|D\kappa (x_t)-D\kappa (x_{\pi_\delta(t)})\right|\le |D^2\kappa(z)||x_t-x_{\pi_\delta(t)}|,
\end{equation*}
where $z \in \BR$ is a point lying on the line segment joining $x_{\pi_\delta(t)}$ and $x_{t}.$ Now, boundedness of the second-order derivatives of $\kappa$ (Assumption \ref{A:Poly-growth}) and Lemma \ref{L:Sampling-Difference} for the term $|x_s-x_{\pi_\delta(s)}|^p$ complete the proof of the lemma.
\end{proof}
We now start the preparation of proving Lemma \ref{L:M-1-S-1} through the helpful results: Lemmas \ref{L:R-1-Rem}, \ref{L:R-2-3-Rem}, and \ref{L:M-1-R-term} mentioned below. Lemma \ref{L:M-1-S-1} was used in the proof of Lemma \ref{L:M1-estimate}.

\begin{lemma}\label{L:R-1-Rem}
Let $x_{t}$ be the solution of equation \eqref{E:det-sys} and $$\mathscr{R}_1^\delta(t)\triangleq \sum_{i=0}^{\lfloor \frac{t}{\delta}\rfloor-1}\int_{i\delta}^{(i+1)\delta}D\kappa(x_{i\delta})f(x_{i\delta})\int_{s}^{(i+1)\delta}\frac{d}{du}e^{\int_u^t[Df(x_r)+D\kappa(x_r)]\thinspace dr}\thinspace du \thinspace ds.$$ Then, for any integer $p\ge 1,$ there exists a positive constant $C_{\ref{L:R-1-Rem}}$ such that
$$\sup_{t \ge 0}|\mathscr{R}_1^\delta(t)|^p \le \delta^p  C_{\ref{L:R-1-Rem}}.$$
\end{lemma}
\begin{proof}
Using the boundedness of the derivatives of $\kappa$, polynomial growth of the function $f$ (Assumptions \ref{A:Poly-Lip}, \ref{A:Poly-growth}), and $\frac{d}{du}e^{\int_u^t[Df(x_r)+D\kappa(x_r)]\thinspace dr}=-[Df(x_u)+D\kappa(x_u)]e^{\int_u^t[Df(x_r)+D\kappa(x_r)]\thinspace dr},$ we have

\begin{equation}\label{E:R-1-delta}
\begin{aligned}
|\mathscr{R}_1^\delta(t)|^p &\le C\left(\sup_{t\ge 0}|x_t|^q +1\right)^p  \left[\sum_{i=0}^{\lfloor \frac{t}{\delta}\rfloor-1}\int_{i\delta}^{(i+1)\delta}\int_{s}^{(i+1)\delta}\left|Df(x_u)+D\kappa(x_u)\right|e^{\int_u^t[Df(x_r)+D\kappa(x_r)]\thinspace dr}\thinspace du \thinspace ds\right]^p\\
& \qquad \qquad \qquad \qquad   \le C\left(\sup_{t\ge 0}|x_t|^{q}+1\right)^{2p} \left[\sum_{i=0}^{\lfloor \frac{t}{\delta}\rfloor-1}\int_{i\delta}^{(i+1)\delta}\int_{s}^{(i+1)\delta}e^{\int_u^t[Df(x_r)+D\kappa(x_r)]\thinspace dr}\thinspace du \thinspace ds\right]^p.
\end{aligned}
\end{equation}

We now focus to get an estimate on the term $\sum_{i=0}^{\lfloor \frac{t}{\delta}\rfloor-1}\int_{i\delta}^{(i+1)\delta}\int_{s}^{(i+1)\delta}e^{\int_u^t[Df(x_r)+D\kappa(x_r)]\thinspace dr}\thinspace du \thinspace ds$ in the above equation. To do this, we note $i\delta \le s \le u  \le (i+1)\delta \le t $, for any $i\in \{0,\cdots, \lfloor \frac{t}{\delta}\rfloor-1\}$ and $e^{\int_u^t[Df(x_r)+D\kappa(x_r)]\thinspace dr}= e^{\int_s^t[Df(x_r)+D\kappa(x_r)]\thinspace dr}e^{-\int_s^u[Df(x_r)+D\kappa(x_r)]\thinspace dr}.$ Hence,
\begin{equation}\label{E:R_1-CLT-Intermediate}
\begin{aligned}
\sum_{i=0}^{\lfloor \frac{t}{\delta}\rfloor-1}\int_{i\delta}^{(i+1)\delta}\int_{s}^{(i+1)\delta} & e^{\int_u^t[Df(x_r)+D\kappa(x_r)]\thinspace dr}\thinspace du \thinspace ds \\
 &= \sum_{i=0}^{\lfloor \frac{t}{\delta}\rfloor-1}\int_{i\delta}^{(i+1)\delta}e^{\int_s^t[Df(x_r)+D\kappa(x_r)]\thinspace dr}\int_{s}^{(i+1)\delta}e^{-\int_s^u[Df(x_r)+D\kappa(x_r)]\thinspace dr}\thinspace du \thinspace ds.
\end{aligned}
\end{equation}
Now, if for any $r \in [s,(i+1)\delta],$ $[Df(x_r)+D\kappa(x_r)] \ge 0,$ then we have $e^{-\int_s^u[Df(x_r)+D\kappa(x_r)]\thinspace dr}\le 1;$ again if for any $r \in [s,(i+1)\delta],$ $[Df(x_r)+D\kappa(x_r)] \le 0,$ we have $e^{-\int_s^u[Df(x_r)+D\kappa(x_r)]\thinspace dr}\le C.$ Indeed, for the latter, $e^{-\int_s^u[Df(x_r)+D\kappa(x_r)]\thinspace dr} \le e^{\int_s^u C dr} \le e^{C\delta}$ as $0 \le u-s \le \delta.$ Hence, $e^{-\int_s^u[Df(x_r)+D\kappa(x_r)]\thinspace dr}\le \max\{1, C\}.$ Returning to equation \eqref{E:R_1-CLT-Intermediate} and using $(i+1)\delta-s \le \delta$, we obtain
\begin{equation}\label{E:R_1-CLT-3}
\begin{aligned}
\sum_{i=0}^{\lfloor \frac{t}{\delta}\rfloor-1}\int_{i\delta}^{(i+1)\delta}\int_{s}^{(i+1)\delta}  e^{\int_u^t[Df(x_r)+D\kappa(x_r)]\thinspace dr}\thinspace du \thinspace ds
 & \le C \delta \sum_{i=0}^{\lfloor \frac{t}{\delta}\rfloor-1}\int_{i\delta}^{(i+1)\delta}e^{\int_s^t[Df(x_r)+D\kappa(x_r)]\thinspace dr} \thinspace ds \\
 & \le C \delta \int_0^t e^{\int_s^t[Df(x_r)+D\kappa(x_r)]\thinspace dr} \thinspace ds.
\end{aligned}
\end{equation}
Finally, combining the equations \eqref{E:R-1-delta}, \eqref{E:R_1-CLT-3} followed by the Assumption \ref{A:Poly-Int}, we obtain the required result.
\end{proof}

\begin{lemma}\label{L:R-2-3-Rem}
Let $x_{t}$ be the solution of equation \eqref{E:det-sys} and
\begin{equation*}
\begin{aligned}
\mathscr{R}_2^\delta(t) &\triangleq \sum_{i=0}^{\lfloor \frac{t}{\delta}\rfloor-1}\int_{i\delta}^{(i+1)\delta}\left[D\kappa(x_{i\delta})f(x_{i\delta})e^{\int_s^t[Df(x_u)+D\kappa (x_u)]du}-D\kappa(x_s)f(x_{i\delta})e^{\int_s^t[Df(x_u)+D\kappa (x_u)]du}\right]ds,\\
\mathscr{R}_3^\delta(t) &\triangleq \sum_{i=0}^{\lfloor \frac{t}{\delta}\rfloor-1}\int_{i\delta}^{(i+1)\delta}\left[ D\kappa(x_s)f(x_{i\delta})e^{\int_s^t[Df(x_u)+D\kappa (x_u)]du}
- D\kappa(x_s)f(x_s)e^{\int_s^t[Df(x_u)+D\kappa (x_u)]du}\right]ds.
\end{aligned}
\end{equation*}
Then, for any integer $p\ge 1,$ there exists a positive constant $C_{\ref{L:R-2-3-Rem}}$ such that
$$\sup_{t \ge 0}|\mathscr{R}_2^\delta(t)|^p+\sup_{t \ge 0}|\mathscr{R}_3^\delta(t)|^p \le \delta^p  C_{\ref{L:R-2-3-Rem}}.$$
\end{lemma}

\begin{proof}
Using Lemma \ref{L:Sec-der-Kappa} to handle term $|D\kappa(x_{i\delta})-D\kappa(x_s)|$, local Lipschitz continuity of $f$, boundedness of the derivative of $\kappa$, Assumption \ref{A:Poly-Int} and Lemma \ref{L:Sampling-Difference} for the term $|x_s-x_{\pi_\delta(s)}|^p,$ we obtain
\begin{equation*}
\begin{aligned}
|\mathscr{R}_2^\delta(t)|^p &\le C \left(1+\sup_{t\ge 0}|x_t|\right)^p\left[ \sum_{i=0}^{\lfloor \frac{t}{\delta}\rfloor-1}\int_{i\delta}^{(i+1)\delta}|D\kappa(x_{i\delta})-D\kappa(x_s)|e^{\int_s^t[Df(x_u)+D\kappa (x_u)]du}\thinspace ds\right]^p \le \delta^p  C, \\
|\mathscr{R}_3^\delta(t)|^p& \le C \left[\sum_{i=0}^{\lfloor \frac{t}{\delta}\rfloor-1}\int_{i\delta}^{(i+1)\delta}|f(x_{i\delta})-f(x_s)|e^{\int_s^t[Df(x_u)+D\kappa (x_u)]du}\thinspace ds\right]^p \le \delta^p  C.
\end{aligned}
\end{equation*}
Putting these estimates together, we obtain the required result.
\end{proof}

\begin{lemma}\label{L:M-1-R-term}
Let $x_{t}$ be the solution of equation \eqref{E:det-sys} and
$${\mathscr{R}}_t^{\eps,\delta}\triangleq \frac{\delta}{2\eps} \left[\sum_{i=0}^{\lfloor \frac{t}{\delta}\rfloor-1}\delta D\kappa(x_{i\delta})f(x_{i\delta})e^{\int_{(i+1)\delta}^t[Df(x_u)+D\kappa (x_u)]du}- \int_0^t D\kappa(x_s)f(x_s)e^{\int_s^t[Df(x_u)+D\kappa (x_u)]du}\thinspace ds\right].$$
Then, for any integer $p\ge 1,$ there exists a positive constant $C_{\ref{L:M-1-R-term}}$ such that
\begin{equation*}
\sup_{t\ge 0}|\mathscr{R}_t^{\eps,\delta}|^p \le \frac{\delta^{2p}}{\eps^p}  C_{\ref{L:M-1-R-term}}.
\end{equation*}
\end{lemma}
\begin{proof}
In the expression of $\mathscr{R}_t^{\eps,\delta}$ above, writing the integral $\int_0^t D\kappa(x_s)f(x_s)e^{\int_s^t[Df(x_u)+D\kappa (x_u)]du}\thinspace ds$ in the Riemann sum form, we have
\begin{equation}\label{E:Rem-CLT}
\begin{aligned}
&\sum_{i=0}^{\lfloor \frac{t}{\delta}\rfloor-1}\delta D\kappa(x_{i\delta})f(x_{i\delta})e^{\int_{(i+1)\delta}^t[Df(x_u)+D\kappa (x_u)]du}- \int_0^t D\kappa(x_s)f(x_s)e^{\int_s^t[Df(x_u)+D\kappa (x_u)]du}\thinspace ds\\
&= \sum_{i=0}^{\lfloor \frac{t}{\delta}\rfloor-1}\int_{i\delta}^{(i+1)\delta} \left[D\kappa(x_{i\delta})f(x_{i\delta})e^{\int_{(i+1)\delta}^t[Df(x_u)+D\kappa (x_u)]du}- D\kappa(x_s)f(x_s)e^{\int_s^t[Df(x_u)+D\kappa (x_u)]du}
\right] ds\\
& \qquad \qquad \qquad \qquad \qquad \qquad \quad \qquad \qquad \qquad \qquad - \int_{\delta\lfloor \frac{t}{\delta}\rfloor}^{t}D\kappa(x_s)f(x_s)e^{\int_s^t[Df(x_u)+D\kappa (x_u)]du}\thinspace ds.
\end{aligned}
\end{equation}
Now, for each $i\in \{0,1,\cdots,\lfloor \frac{t}{\delta}\rfloor-1\}$, using a simple algebra for the integrand of the first term on the right-hand side of the above equation, we get
\begin{equation*}
\begin{aligned}
D\kappa(x_{i\delta})&f(x_{i\delta})e^{\int_{(i+1)\delta}^t[Df(x_u)+D\kappa (x_u)]du}- D\kappa(x_s)f(x_s)e^{\int_s^t[Df(x_u)+D\kappa (x_u)]du}\\ &= D\kappa(x_{i\delta})f(x_{i\delta})e^{\int_{(i+1)\delta}^t[Df(x_u)+D\kappa (x_u)]du}
-D\kappa(x_s)f(x_{i\delta})e^{\int_s^t[Df(x_u)+D\kappa (x_u)]du}\\
&\qquad \qquad \qquad  +D\kappa(x_s)f(x_{i\delta})e^{\int_s^t[Df(x_u)+D\kappa (x_u)]du}
- D\kappa(x_s)f(x_s)e^{\int_s^t[Df(x_u)+D\kappa (x_u)]du}.
\end{aligned}
\end{equation*}
Thus,
\begin{equation}\label{E:Rem-CLT-1}
\begin{aligned}
D\kappa(x_{i\delta})&f(x_{i\delta})e^{\int_{(i+1)\delta}^t[Df(x_u)+D\kappa (x_u)]du}- D\kappa(x_s)f(x_s)e^{\int_s^t[Df(x_u)+D\kappa (x_u)]du}\\
&=D\kappa(x_{i\delta})f(x_{i\delta})e^{\int_{(i+1)\delta}^t[Df(x_u)+D\kappa (x_u)]du}- D\kappa(x_{i\delta})f(x_{i\delta})e^{\int_s^t[Df(x_u)+D\kappa (x_u)]du}\\
&\qquad \qquad \quad+ D\kappa(x_{i\delta})f(x_{i\delta})e^{\int_s^t[Df(x_u)+D\kappa (x_u)]du}-D\kappa(x_s)f(x_{i\delta})e^{\int_s^t[Df(x_u)+D\kappa (x_u)]du}\\
&\qquad \qquad \qquad  +D\kappa(x_s)f(x_{i\delta})e^{\int_s^t[Df(x_u)+D\kappa (x_u)]du}
- D\kappa(x_s)f(x_s)e^{\int_s^t[Df(x_u)+D\kappa (x_u)]du}\\
&=D\kappa(x_{i\delta})f(x_{i\delta})\int_{s}^{(i+1)\delta}\frac{d}{du}e^{\int_u^t[Df(x_r)+D\kappa (x_r)]dr}\thinspace du\\
&\qquad \qquad \qquad+ D\kappa(x_{i\delta})f(x_{i\delta})e^{\int_s^t[Df(x_u)+D\kappa (x_u)]du}-D\kappa(x_s)f(x_{i\delta})e^{\int_s^t[Df(x_u)+D\kappa (x_u)]du}\\
&\qquad \qquad \qquad+D\kappa(x_s)f(x_{i\delta})e^{\int_s^t[Df(x_u)+D\kappa (x_u)]du}
- D\kappa(x_s)f(x_s)e^{\int_s^t[Df(x_u)+D\kappa (x_u)]du}.
\end{aligned}\end{equation}
Hence, for any integer $p\ge 1,$ from \eqref{E:Rem-CLT} and \eqref{E:Rem-CLT-1}, we get
\begin{equation*}
\begin{aligned}
&|\mathscr{R}_t^{\eps,\delta}|^p \le \frac{\delta^p}{\eps^p}\left|\sum_{i=0}^{\lfloor \frac{t}{\delta}\rfloor-1}\delta D\kappa(x_{i\delta})f(x_{i\delta})e^{\int_{(i+1)\delta}^t[Df(x_u)+D\kappa (x_u)]du}- \int_0^t D\kappa(x_s)f(x_s)e^{\int_s^t[Df(x_u)+D\kappa (x_u)]du}\thinspace ds\right|^p \\
&\quad  \le C\left|\sum_{i=0}^{\lfloor \frac{t}{\delta}\rfloor-1}\int_{i\delta}^{(i+1)\delta} \left[D\kappa(x_{i\delta})f(x_{i\delta})e^{\int_{(i+1)\delta}^t[Df(x_u)+D\kappa (x_u)]du}- D\kappa(x_s)f(x_s)e^{\int_s^t[Df(x_u)+D\kappa (x_u)]du}
\right] ds \right|^p\\
& \qquad \qquad \qquad \qquad \qquad \qquad \qquad \qquad \qquad \qquad \qquad    + C\left|\int_{\delta\lfloor \frac{t}{\delta}\rfloor}^{t}D\kappa(x_s)f(x_s)e^{\int_s^t[Df(x_u)+D\kappa (x_u)]du}\thinspace ds \right|^p.
\end{aligned}
\end{equation*}
Therefore,
\begin{equation*}
\begin{aligned}
&|\mathscr{R}_t^{\eps,\delta}|^p \le C\left|\sum_{i=0}^{\lfloor \frac{t}{\delta}\rfloor-1}\int_{i\delta}^{(i+1)\delta}D\kappa(x_{i\delta})f(x_{i\delta})\int_{s}^{(i+1)\delta}\frac{d}{du}e^{\int_u^t[Df(x_r)+D\kappa (x_r)]dr}\thinspace du \thinspace ds  \right|^p\\
&\quad     + C\left|\sum_{i=0}^{\lfloor \frac{t}{\delta}\rfloor-1}\int_{i\delta}^{(i+1)\delta}\left[D\kappa(x_{i\delta})f(x_{i\delta})e^{\int_s^t[Df(x_u)+D\kappa (x_u)]du}-D\kappa(x_s)f(x_{i\delta})e^{\int_s^t[Df(x_u)+D\kappa (x_u)]du}\right]ds \right|^p\\
\end{aligned}
\end{equation*}
\begin{equation*}
\begin{aligned}
&\quad   + C\left|\sum_{i=0}^{\lfloor \frac{t}{\delta}\rfloor-1}\int_{i\delta}^{(i+1)\delta}\left[ D\kappa(x_s)f(x_{i\delta})e^{\int_s^t[Df(x_u)+D\kappa (x_u)]du}
- D\kappa(x_s)f(x_s)e^{\int_s^t[Df(x_u)+D\kappa (x_u)]du}\right]ds\right|^p\\
& \qquad \qquad \qquad \qquad \qquad \qquad \qquad \qquad   +C \left|\int_{\delta\lfloor \frac{t}{\delta}\rfloor}^{t}D\kappa(x_s)f(x_s)e^{\int_s^t[Df(x_u)+D\kappa (x_u)]du}\thinspace ds \right|^p\\
& \qquad \qquad \qquad   \triangleq C\left(\left|\mathscr{R}_1^\delta(t)\right|^p + \left|\mathscr{R}_2^\delta(t)\right|^p+ \left|\mathscr{R}_3^\delta(t)\right|^p+\left|\int_{\delta\lfloor \frac{t}{\delta}\rfloor}^{t}D\kappa(x_s)f(x_s)e^{\int_s^t[Df(x_u)+D\kappa (x_u)]du}\thinspace ds \right|^p\right).
\end{aligned}
\end{equation*}
We now use the condition $|\int_{\delta\lfloor \frac{t}{\delta}\rfloor}^{t}e^{\int_s^t[Df(x_u)+D\kappa (x_u)]du}\thinspace ds|<\delta$, boundedness of the derivative of $\kappa$ and polynomial growth of the function $f$ for the last term in the above equation. Lemmas \ref{L:R-1-Rem}, \ref{L:R-2-3-Rem}, Assumption \ref{A:Poly-Int} and the fact that $t-\delta\lfloor \frac{t}{\delta}\rfloor < \delta$ for any $\delta>0$ yield the required result.
\end{proof}

We are now in the position to prove Lemma \ref{L:M-1-S-1}.
\begin{lemma}\label{L:M-1-S-1}
Let $x_{t}$ be the solution of equation \eqref{E:det-sys} and
$$\mathscr{S}_1^{\eps,\delta}(t) \triangleq  \int_0^t \left[(1/\eps) D\kappa (x_{\pi_\delta(s)}) f(x_{\pi_\delta(s)})[{s-\pi_\delta(s)}]-\frac{\delta}{2 \eps}D\kappa(x_s)f(x_s)\right]e^{\int_s^t[Df(x_u)+D\kappa (x_u)]du}\thinspace ds.$$
Then, for any integer $p\ge 1,$ we have a positive constant $C_{\ref{L:M-1-S-1}}$ such that
$$\sup_{t \ge 0}|\mathscr{S}_1^{\eps,\delta}(t)|^p\le \frac{\delta^{2p}}{\eps^p}  C_{\ref{L:M-1-S-1}}.$$
\end{lemma}
\begin{proof}
We recall
$$\mathscr{S}_1^{\eps,\delta}(t)= \int_0^t \left[(1/\eps) D\kappa (x_{\pi_\delta(s)}) f(x_{\pi_\delta(s)})[{s-\pi_\delta(s)}]-\frac{\delta}{2 \eps}D\kappa(x_s)f(x_s)\right]e^{\int_s^t[Df(x_u)+D\kappa (x_u)]du}\thinspace ds.$$
Moving in the direction of getting an estimate for the quantity $\sup_{t \ge 0}|\mathscr{S}_1^{\eps,\delta}(t)|^p$, we first write the integral (below) in the Riemann sum form followed by using integration by parts formula to obtain
\begin{equation*}
\begin{aligned}
\int_0^t&(1/\eps) D\kappa (x_{\pi_\delta(s)}) f(x_{\pi_\delta(s)})[{s-\pi_\delta(s)}]e^{\int_s^t[Df(x_u)+D\kappa (x_u)]du}\thinspace ds\\
&\qquad \qquad \quad=\sum_{i=0}^{\lfloor \frac{t}{\delta}\rfloor-1}\frac{1}{\eps}D\kappa(x_{i\delta})f(x_{i\delta})\int_{i\delta}^{(i+1)\delta}(s-i\delta)e^{\int_s^t[Df(x_u)+D\kappa (x_u)]du}\thinspace ds\\
& \qquad \qquad \qquad \qquad \qquad \qquad \quad +\frac{1}{\eps}D\kappa(x_{\pi_\delta(t)})f(x_{\pi_\delta(t)})\int_{\delta\lfloor \frac{t}{\delta}\rfloor}^{t}[s-\pi_\delta(s)]e^{\int_s^t[Df(x_u)+D\kappa (x_u)]du}\thinspace ds\\
&\qquad \qquad \quad=\sum_{i=0}^{\lfloor \frac{t}{\delta}\rfloor-1}\frac{1}{\eps}D\kappa(x_{i\delta})f(x_{i\delta})\frac{\delta^2}{2}e^{\int_{(i+1)\delta}^t[Df(x_u)+D\kappa (x_u)]du}\\
&\qquad \qquad \quad  +\sum_{i=0}^{\lfloor \frac{t}{\delta}\rfloor-1}\frac{1}{2\eps}D\kappa(x_{i\delta})f(x_{i\delta})\int_{i\delta}^{(i+1)\delta}(s-i\delta)^2 [Df(x_s)+D\kappa (x_s)] e^{\int_s^t[Df(x_u)+D\kappa (x_u)]du}\thinspace ds\\
&\qquad \qquad \qquad \qquad \qquad \qquad \quad \quad +\frac{1}{\eps}D\kappa(x_{\pi_\delta(t)})f(x_{\pi_\delta(t)})\int_{\delta\lfloor \frac{t}{\delta}\rfloor}^{t}[s-\pi_\delta(s)]e^{\int_s^t[Df(x_u)+D\kappa (x_u)]du}\thinspace ds.
\end{aligned}
\end{equation*}
Therefore,
\begin{equation*}
\begin{aligned}
 \mathscr{S}_1^{\eps,\delta}(t)& = \frac{\delta}{2\eps} \left[\sum_{i=0}^{\lfloor \frac{t}{\delta}\rfloor-1}\delta D\kappa(x_{i\delta})f(x_{i\delta})e^{\int_{(i+1)\delta}^t[Df(x_u)+D\kappa (x_u)]du}- \int_0^t D\kappa(x_s)f(x_s)e^{\int_s^t[Df(x_u)+D\kappa (x_u)]du}\thinspace ds\right]\\
&\qquad \qquad \quad   +\sum_{i=0}^{\lfloor \frac{t}{\delta}\rfloor-1}\frac{1}{2\eps}D\kappa(x_{i\delta})f(x_{i\delta})\int_{i\delta}^{(i+1)\delta}(s-i\delta)^2 [Df(x_s)+D\kappa (x_s)] e^{\int_s^t[Df(x_u)+D\kappa (x_u)]du}\thinspace ds\\
&\qquad \qquad \qquad \qquad \qquad \qquad \quad +\frac{1}{\eps}D\kappa(x_{\pi_\delta(t)})f(x_{\pi_\delta(t)})\int_{\delta\lfloor \frac{t}{\delta}\rfloor}^{t}[s-\pi_\delta(s)]e^{\int_s^t[Df(x_u)+D\kappa (x_u)]du}\thinspace ds.
\end{aligned}
\end{equation*}
Now, recalling the definition of $\mathscr{R}_t^{\eps,\delta}$ from Lemma \ref{L:M-1-R-term}, we have
\begin{equation*}
\begin{aligned}
\left|\mathscr{S}_1^{\eps,\delta}(t) \right|^p & \le C|\mathscr{R}_t^{\eps,\delta}|^p
 + \frac{\delta^p}{\eps^p}C \left(1+\sup_{t\ge 0}|x_t|^q \right) \delta^p \left(\int_0^{\delta\lfloor \frac{t}{\delta}\rfloor}|Df(x_s)+D\kappa (x_s)| e^{\int_s^t[Df(x_u)+D\kappa (x_u)]du}\thinspace ds\right)^p \\
 & \qquad \qquad \qquad \qquad \qquad \qquad \quad  \qquad  + \frac{\delta^p}{\eps^p}C \left(1+\sup_{t\ge 0}|x_t|^q \right) \left(\int_{\delta\lfloor \frac{t}{\delta}\rfloor}^{t}e^{\int_s^t[Df(x_u)+D\kappa (x_u)]du}\thinspace ds\right)^p.
\end{aligned}
\end{equation*}
The last equation is obtained by using the boundedness of derivatives of $\kappa$ and the polynomial growth of $f$. Employing Lemma \ref{L:M-1-R-term} for the term $|\mathscr{R}_t^{\eps,\delta}|^p$, Assumption \ref{A:Poly-Int} for the term $\int_0^t e^{\int_s^t[Df(x_u)+D\kappa (x_u)]du}\thinspace ds$, assuming the condition $|\int_{\delta\lfloor \frac{t}{\delta}\rfloor}^{t}e^{\int_s^t[Df(x_u)+D\kappa (x_u)]du}\thinspace ds|<\delta$ with a note $t-\delta\lfloor \frac{t}{\delta}\rfloor < \delta$ for any $\delta>0$, we obtain the required estimate.
\end{proof}

Next, we state and prove two supporting Lemmas \ref{L:M-2-S-2} and \ref{L:M-2-Intermediate}; these results were used in the proof of Lemma \ref{L:M2-estimate}.

\begin{lemma}\label{L:M-2-S-2}
Let $x_{t}$ be the solution of equation \eqref{E:det-sys} and
$$\mathscr{S}_2^{\eps,\delta}(t) \triangleq  \int_0^t \left[(1/\eps) D\kappa (x_{\pi_\delta(s)}) \kappa(x_{\pi_\delta(s)})[{s-\pi_\delta(s)}]-\frac{\delta}{2 \eps}D\kappa(x_s)\kappa(x_s)\right]e^{\int_s^t[Df(x_u)+D\kappa (x_u)]du}\thinspace ds.$$
Then, for any integer $p\ge 1,$ we have a positive constant $C_{\ref{L:M-2-S-2}}$ such that
$$\sup_{t \ge 0}|\mathscr{S}_2^{\eps,\delta}(t)|^p\le \frac{\delta^{2p}}{\eps^p} C_{\ref{L:M-2-S-2}}.$$
\end{lemma}

\begin{proof}
The proof follows by similar calculations to the proof of Lemma \ref{L:M-1-S-1}.
\end{proof}

\begin{lemma}\label{L:M-2-Intermediate}
Let $x_{t}$ and $X_{t}^{\varepsilon,\delta}$ solve \eqref{E:det-sys} and \eqref{E:W-S-SDE-With-pert}, respectively and $p\ge 1$ be an integer, then, for all sufficiently small $\eps,\delta>0,$ there exists a positive constant $C_{\ref{L:M-2-Intermediate}}$ such that
\begin{equation*}
\sup_{t \ge 0}\BE\left[ \int_0^t \left|X_{\pi_\delta(s)}^{\eps,\delta}-x_{\pi_\delta(s)}\right|e^{\int_s^t[Df(x_u)+D\kappa (x_u)]du}\thinspace ds\right]^p \le C_{\ref{L:M-2-Intermediate}}(\delta^p+\eps^p+{\delta}^{\frac{p}{2}}\eps^p).
\end{equation*}
\end{lemma}
\begin{proof}
For any integer $p\ge 1$, we start by noting
\begin{equation*}
\begin{aligned}
&\left[ \int_0^t  \left|X_{\pi_\delta(s)}^{\eps,\delta}-x_{\pi_\delta(s)}\right|e^{\int_s^t[Df(x_u)+D\kappa (x_u)]du}   \thinspace ds\right]^p \\
& \qquad \qquad \qquad \qquad \qquad  =\int_{[0,t]^p}\left( \prod_{i=1}^p e^{\int_{s_i}^t[Df(x_u)+D\kappa (x_u)]du}\right)\left( \prod_{i=1}^p  \left|X_{\pi_\delta(s_i)}^{\eps,\delta}-x_{\pi_\delta(s_i)}\right| \right)ds_1 \cdots ds_p \\
& \qquad \qquad \qquad \qquad \qquad  \le C\int_{[0,t]^p}\left( \prod_{i=1}^p e^{\int_{s_i}^t[Df(x_u)+D\kappa (x_u)]du}\right)\left( \sum_{i=1}^p  \left|X_{\pi_\delta(s_i)}^{\eps,\delta}-x_{\pi_\delta(s_i)}\right|^p \right)ds_1 \cdots ds_p.
\end{aligned}
\end{equation*}
Taking expectation on both sides of the above equation followed by Theorem \ref{T:LLN-UIT}, we get
\begin{equation*}
\begin{aligned}
\BE &\left[ \int_0^t  \left|X_{\pi_\delta(s)}^{\eps,\delta}-x_{\pi_\delta(s)}\right|e^{\int_s^t[Df(x_u)+D\kappa (x_u)]du}   \thinspace ds\right]^p \\
 & \qquad \qquad \qquad \qquad \qquad \quad \le  C\int_{[0,t]^p} \left(\prod_{i=1}^p e^{\int_{s_i}^t[Df(x_u)+D\kappa (x_u)]du}\right) \sum_{i=1}^p  \BE \left|X_{\pi_\delta(s_i)}^{\eps,\delta}-x_{\pi_\delta(s_i)}\right|^p ds_1 \cdots ds_p\\
& \qquad \qquad \qquad \qquad \qquad \quad \qquad \qquad \quad  \le C(\delta^p+\eps^p+{\delta}^{\frac{p}{2}}\eps^p)\int_{[0,t]^p} \prod_{i=1}^p e^{\int_{s_i}^t[Df(x_u)+D\kappa (x_u)]du}ds_1 \cdots ds_p.
\end{aligned}
\end{equation*}
A simplification of the last integral along with Assumption \ref{A:Poly-Int} gives the required bound.
\end{proof}

\subsection{Proof of Propositions \ref{P:CLT-R-12} and \ref{P:CLT-R-3}}\label{SS:FluctuationsControl}

This section is devoted to the proofs of Propositions \ref{P:CLT-R-12} and \ref{P:CLT-R-3}. These results were used in the proof of Proposition \ref{P:I3-Prop-Rem}.
\begin{proposition}\label{P:CLT-R-12}
Let ${\sf R}_1^{\eps,\delta}(t)$ and ${\sf R}_2^{\eps,\delta}(t)$  be defined as in equation \eqref{E:Remainder-terms} and $p\ge 1$ be an integer. Then, for all sufficiently small $\eps,\delta>0,$ there exists a positive constant $C_{\ref{P:CLT-R-12}}$ such that
\begin{equation*}
\sup_{t \ge 0}\BE\left|\int_0^t e^{\int_s^t[Df(x_u)+D\kappa (x_u)]du} {\sf R}^{\eps,\delta}_1(s)ds\right|^p \le \frac{C_{\ref{P:CLT-R-12}}}{\eps^{p}}(\delta^{2p}+\eps^{2p}+{\delta}^{p}\eps^{2p}),
\end{equation*}
\begin{equation*}
\sup_{t \ge 0}\BE \left|\int_0^t e^{\int_s^t[Df(x_u)+D\kappa (x_u)]du} {\sf R}^{\eps,\delta}_2(s)ds\right|^p\le \frac{C_{\ref{P:CLT-R-12}}}{\eps^{p}}(\delta^{2p}+\eps^{2p}+{\delta}^{p}\eps^{2p}).
\end{equation*}
\end{proposition}

\begin{proof}
Using Taylor's theorem and the polynomial growth of the function $f$, we have
\begin{equation}\label{E:R1-R2-Est}
|{\sf R}^{\eps,\delta}_1(s)|^p \triangleq \left|\frac{f(X_s^{\eps,\delta})-f(x_s)}{\eps}-Df(x_s) Z_s^{\eps,\delta}\right|^p= \frac{1}{2^p\eps^p}\left| D^2f(z)(X_s^{\eps,\delta}-x_s)^2 \right|^p, 
\end{equation}
where $z \in \BR$ is a (random) point lying on the line segment joining $x_s$ and $X_s^{\eps,\delta}$. Hence, there is some integer $r<\infty$ so that H\"{o}lder's inequality together with the bounds from Lemma \ref{L:Poly-p-Moment} give
\begin{align*}
\BE|{\sf R}^{\eps,\delta}_1(s)|^p \leq \frac{C}{\eps^p}\left(\BE\left| z \right|^{4r p} +1\right)^{1/2} \left(\BE\left| X_s^{\eps,\delta}-x_s \right|^{4p} \right)^{1/2} \le \frac{C}{\eps^p}\left(\BE\left|X_s^{\eps,\delta}-x_s \right|^{4p}\right)^{1/2},
\end{align*}
for some time-independent constant $C<\infty$. Next, for any integer $p\ge 1,$ we have
\begin{equation*}
\begin{aligned}
&\left|\int_0^te^{\int_s^t[Df(x_u)+D\kappa (x_u)]du} {\sf R}^{\eps,\delta}_1(s)ds\right|^p \le \left(\int_0^t e^{\int_{s}^t[Df(x_u)+D\kappa (x_u)]du}\left|{\sf R}^{\eps,\delta}_1(s)\right|ds\right)^p\\
& \qquad \qquad \qquad \qquad \qquad \quad \qquad \qquad \le \int_{[0,t]^p}\left( \prod_{i=1}^p e^{\int_{s_i}^t[Df(x_u)+D\kappa (x_u)]du}\right)\left( \prod_{i=1}^p \left|{\sf R}^{\eps,\delta}_1(s_i)\right| \right)ds_1 \cdots ds_p \\
& \qquad \qquad \qquad \qquad \qquad \qquad \qquad \quad  \le \int_{[0,t]^p}\left( \prod_{i=1}^p e^{\int_{s_i}^t[Df(x_u)+D\kappa (x_u)]du}\right)\left( \sum_{i=1}^p \left|{\sf R}^{\eps,\delta}_1(s_i)\right|^p \right)ds_1 \cdots ds_p.
\end{aligned}
\end{equation*}
Now, taking expectation, using equation \eqref{E:R1-R2-Est},  the local Lipschitz condition \ref{A:Poly-Lip} of $f$ and H\"{o}lder's inequality, we have
\begin{align*}
&\BE\left[\left|\int_0^t e^{\int_s^t[Df(x_u)+D\kappa (x_u)]du} {\sf R}^{\eps,\delta}_1(s)ds\right|^p\right]\\
& \qquad \qquad \qquad \qquad \quad  \le \int_{[0,t]^p}\left( \prod_{i=1}^p e^{\int_{s_i}^t[Df(x_u)+D\kappa (x_u)]du} \right) \sum_{i=1}^p \BE\left[\left|{\sf R}^{\eps,\delta}_1(s_i)\right|^p \right] ds_1 \cdots ds_p \\
 &\qquad \qquad \qquad \qquad \quad \le \frac{C}{\eps^p}\int_{[0,t]^p}\left( \prod_{i=1}^p e^{\int_{s_i}^t[Df(x_u)+D\kappa (x_u)]du}\right)\sum_{i=1}^p \left(\BE\left[\left|X_{s_i}^{\eps,\delta}-x_{s_i} \right|^{4p}\right]^{1/2} \right) ds_1 \cdots ds_p.
\end{align*}

Employing Theorem \ref{T:LLN-UIT}, Lemma \ref{L:Poly-p-Moment} and Assumption \ref{A:Poly-Int} in the above equation, we get the required result. Similarly, following the same steps and recalling the definition of ${\sf R}_2^{\eps,\delta}(t)$ from equation \eqref{E:Remainder-terms}, we have
\begin{equation*}
\sup_{t \ge 0}\BE\left|\int_0^t e^{\int_s^t[Df(x_u)+D\kappa (x_u)]du} {\sf R}^{\eps,\delta}_2(s)ds\right|^p\le \frac{C}{\eps^{p}}(\delta^{2p}+\eps^{2p} +{\delta}^{p}\eps^{2p}).
\end{equation*}
\end{proof}

Before proceeding with the proof of Proposition \ref{P:CLT-R-3}, we find it convenient to state and prove Lemma \ref{L:R-3-interMedTerm} below.

\begin{lemma}\label{L:R-3-interMedTerm}
Let $x_{t}$ and $X_{t}^{\varepsilon,\delta}$ solve \eqref{E:det-sys} and \eqref{E:W-S-SDE-With-pert}, respectively. For any integer $p\ge 1$ and all sufficiently small $\eps,\delta>0$, there exists a positive constant $C_{\ref{L:R-3-interMedTerm}}$ such that
\begin{equation}\label{E:R-3-interMedTerm}
\begin{aligned}
\sup_{u \ge 0}\BE \left|\left(D\kappa(x_u)-D\kappa(X_{\pi_{\delta}(u)}^{\eps,\delta})\right)\left(X_u^{\eps,\delta}-X_{{\pi_\delta(u)}}^{\eps,\delta}\right)\right|^{p} & \le (\eps^{2p}+\delta^{2p}+\eps^{2p} \delta^p)C_{\ref{L:R-3-interMedTerm}},\\
\sup_{u \ge 0} \BE\left|X_u^{\eps,\delta}-X_{\pi_{\delta}(u)}^{\eps,\delta} \right|^{2p} & \le (\eps^{2p}+\delta^{2p}+\eps^{2p} \delta^p)C_{\ref{L:R-3-interMedTerm}}.
\end{aligned}
\end{equation}
\end{lemma}
\begin{proof}
We first prove the first line of equation \eqref{E:R-3-interMedTerm}. For this, we first get an estimate for the term  $D\kappa(x_u)-D\kappa(X_{\pi_{\delta}(u)}^{\eps,\delta}).$ Using Taylor's theorem and boundedness of the second-order derivatives of $\kappa$, we have
\begin{equation*}
\left|D\kappa(X_u^{\eps,\delta})-D\kappa (x_u)\right|\le |D^2\kappa(z)||X_u^{\eps,\delta}-x_{u}|\le C|X_u^{\eps,\delta}-x_{u}|,
\end{equation*}
where $z \in \BR$ is a point (random) lying on the line segment joining $X_u^{\eps,\delta}$ and $x_{u}.$ By the similar arguments, we have $\left|D\kappa (X_u^{\eps,\delta})- D\kappa(X_{\pi_{\delta}(u)}^{\eps,\delta})\right|\le C|X_u^{\eps,\delta}-X_{\pi_{\delta}(u)}^{\eps,\delta}|.$ Further, writing $D\kappa(x_u)-D\kappa(X_{\pi_{\delta}(u)}^{\eps,\delta})$ as $D\kappa(x_u)-D\kappa (X_u^{\eps,\delta})+D\kappa (X_u^{\eps,\delta})-D\kappa(X_{\pi_{\delta}(u)}^{\eps,\delta})$ followed by the triangle inequality, we obtain
\begin{equation}\label{E:R-3-E1}
D\kappa(x_u)-D\kappa(X_{\pi_{\delta}(u)}^{\eps,\delta}) \le C\left(|X_u^{\eps,\delta}-x_{u}|+ |X_u^{\eps,\delta}-X_{\pi_{\delta}(u)}^{\eps,\delta}|\right).
\end{equation}
Now, for the first line in equation \eqref{E:R-3-interMedTerm}, employing equation \eqref{E:R-3-E1} and using a simple algebra, we have
\begin{multline}\label{E:Higher-Power-Der}
\left|\left[D\kappa(x_u)-D\kappa(X_{\pi_{\delta}(u)}^{\eps,\delta})\right]\left(X_u^{\eps,\delta}-X_{{\pi_\delta(u)}}^{\eps,\delta}\right)\right|^p  \le C\left(|X_u^{\eps,\delta}-x_u|^p+\left|X_u^{\eps,\delta}-X_{\pi_{\delta}(u)}^{\eps,\delta}\right|^p\right)\left|X_u^{\eps,\delta}-X_{{\pi_\delta(u)}}^{\eps,\delta}\right|^p\\
\qquad \qquad \qquad \qquad \qquad \qquad \qquad \qquad \quad \le C \left(|X_u^{\eps,\delta}-x_u|^p\left|X_u^{\eps,\delta}-X_{{\pi_\delta(u)}}^{\eps,\delta}\right|^p+\left|X_u^{\eps,\delta}-X_{\pi_{\delta}(u)}^{\eps,\delta}\right|^{2p}\right)\\
\qquad \qquad \qquad \qquad \qquad \qquad \quad  \le C\left(|X_u^{\eps,\delta}-x_u|^{2p}+ |X_u^{\eps,\delta}-x_u|^p\left|x_u-X_{\pi_{\delta}(u)}^{\eps,\delta}\right|^p+\left|X_u^{\eps,\delta}-X_{\pi_{\delta}(u)}^{\eps,\delta}\right|^{2p}\right)\\
\qquad \qquad \qquad \qquad \le C\left(\left|X_u^{\eps,\delta}-x_u\right|^{2p}+ |x_u-x_{\pi_{\delta}(u)}|^{2p}+\left|x_{\pi_{\delta}(u)}-X_{\pi_{\delta}(u)}^{\eps,\delta}\right|^{2p}\right).
\end{multline}
Finally, using Theorem \ref{T:LLN-UIT} and Lemma \ref{L:Sampling-Difference} for the term $\left|x_u-x_{\pi_\delta(u)}\right|^{2p}$, we get
\begin{equation*}
\begin{aligned}
\sup_{u \ge 0}\BE\left[ \left|\left(D\kappa(x_u)-D\kappa(X_{\pi_{\delta}(u)}^{\eps,\delta})\right)\left(X_u^{\eps,\delta}-X_{{\pi_\delta(u)}}^{\eps,\delta}\right)\right|^{p} \right] & \le (\eps^{2p}+\delta^{2p}+\eps^{2p} \delta^p)C.
\end{aligned}
\end{equation*}
The second line of equation \eqref{E:R-3-interMedTerm} is obtained by using Theorem \ref{T:LLN-UIT} and Lemma \ref{L:Sampling-Difference}.
\end{proof}

\begin{proposition}\label{P:CLT-R-3}
Let ${\sf R}_3^{\eps,\delta}(t)$ be defined as in equation \eqref{E:Remainder-terms} and $p\in \BN$. Then, for all sufficiently small $\eps,\delta>0,$ there exists a positive constant $C_{\ref{P:CLT-R-3}}$ such that
\begin{equation*}
\sup_{t \ge 0}\BE\left|\int_0^t e^{\int_s^t[Df(x_u)+D\kappa (x_u)]du} {\sf R}^{\eps,\delta}_3(s)ds\right|^p\le \frac{C_{\ref{P:CLT-R-3}}}{\eps^p}(\eps^{2p}+\delta^{2p}+\eps^{2p} \delta^p).
\end{equation*}
\end{proposition}

\begin{proof}
Using Taylor's theorem followed by a simple algebra, we obtain

\begin{multline*}
\kappa(X_s^{\eps,\delta})=\kappa(X_{\pi_\delta(s)}^{\eps,\delta})+ \left[D\kappa(X_{\pi_\delta(s)}^{\eps,\delta})-D\kappa(x_s)\right]\left(X_s^{\eps,\delta}-X_{\pi_\delta(s)}^{\eps,\delta}\right)
+D\kappa(x_s)\left(X_s^{\eps,\delta}-X_{\pi_{\delta}(s)}^{\eps,\delta}\right)\\
\qquad \qquad \qquad \qquad \qquad \qquad +\frac{1}{2} D^2\kappa(z)\left(X_s^{\eps,\delta}-X_{\pi_\delta(s)}^{\eps,\delta}\right)^2,
\end{multline*}
where $z \in \BR$ is a point (random) lying on the line segment joining $X_{\pi_\delta(s)}^{\eps,\delta}$ and $X_s^{\eps,\delta}$. Now, using the boundedness of the second-order derivatives of $\kappa$, we get  for any integer $p\ge 1$
\begin{multline}\label{E:R3-Est-Taylor}
\left|{\sf R}_3^{\eps,\delta}(s)\right|^p \triangleq \left|\frac{\kappa(X_s^{\eps,\delta})-\kappa(X_{\pi_\delta(s)}^{\eps,\delta})}{\eps}-D\kappa(x_s)\frac{X_s^{\varepsilon, \delta}-X_{\pi_\delta(s)}^{\varepsilon, \delta}}{\eps}\right|^p\\
  \le   \frac{C}{\eps^p}\left|\left[D\kappa(x_s)-D\kappa(X_{\pi_{\delta}(s)}^{\eps,\delta})\right]\left(X_s^{\eps,\delta}-X_{{\pi_\delta(s)}}^{\eps,\delta}\right)\right|^p\\
\qquad \qquad \qquad \qquad \qquad \qquad \qquad \qquad \qquad \qquad \qquad \quad \quad + \frac{C}{\eps^p}\left| D^2\kappa(z)\left(X_s^{\eps,\delta}-X_{\pi_\delta(s)}^{\eps,\delta}\right)^2 \right|^p\\
\qquad \qquad \le  \frac{C}{\eps^p}\left|\left(D\kappa(x_s)-D\kappa(X_{\pi_{\delta}(s)}^{\eps,\delta})\right)\left(X_s^{\eps,\delta}-X_{{\pi_\delta(s)}}^{\eps,\delta}\right)\right|^p + \frac{C}{\eps^p} \left|X_s^{\eps,\delta}-X_{\pi_{\delta}(s)}^{\eps,\delta}\right|^{2p}.
\end{multline}
Next, for any integer $p\ge 1,$ we have
\begin{multline*}
\left|\int_0^t e^{\int_s^t[Df(x_u)+D\kappa (x_u)]du} {\sf R}^{\eps,\delta}_3(s)ds\right|^p \le \left(\int_0^t e^{\int_s^t[Df(x_u)+D\kappa (x_u)]du}\left|{\sf R}^{\eps,\delta}_3(s)\right|ds\right)^p\\
 \le \int_{[0,t]^p}\left( \prod_{i=1}^p e^{\int_{s_i}^t[Df(x_u)+D\kappa (x_u)]du}\right)\left( \sum_{i=1}^p  \left|{\sf R}^{\eps,\delta}_3(s_i)\right|^p \right)ds_1 \cdots ds_p.
\end{multline*}
Taking expectation and using equation \eqref{E:R3-Est-Taylor}, we have
\begin{multline*}
\BE\left|\int_0^t e^{\int_s^t[Df(x_u)+D\kappa (x_u)]du} {\sf R}^{\eps,\delta}_3(s)ds\right|^p \\
\qquad \qquad \le \int_{[0,t]^p}\left( \prod_{i=1}^p e^{\int_{s_i}^t[Df(x_u)+D\kappa (x_u)]du}\right) \sum_{i=1}^p \BE\left[\left|{\sf R}^{\eps,\delta}_3(s_i)\right|^p \right] ds_1 \cdots ds_p \\
\qquad \qquad \qquad \qquad \quad \le \frac{C}{\eps^p}\int_{[0,t]^p}\left( \prod_{i=1}^p e^{\int_{s_i}^t[Df(x_u)+D\kappa (x_u)]du}\right)\sum_{i=1}^p \BE\left[\left|X_{s_i}^{\eps,\delta}-X_{\pi_{\delta}(s_i)}^{\eps,\delta} \right|^{2p} \right]ds_1 \cdots ds_p\\
 + \frac{C}{\eps^p}\int_{[0,t]^p}\left( \prod_{i=1}^p e^{\int_{s_i}^t[Df(x_u)+D\kappa (x_u)]du}\right)\times \\
\sum_{i=1}^p \BE\left[\left|\left(D\kappa(x_{s_i})-D\kappa(X_{\pi_{\delta}(s_i)}^{\eps,\delta})\right)\left(X_{s_i}^{\eps,\delta}-X_{{\pi_\delta(s_i)}}^{\eps,\delta}\right)\right|^{p} \right] ds_1 \cdots ds_p.
 \end{multline*}
Finally, using Lemma \ref{L:R-3-interMedTerm} and Assumption \ref{A:Poly-Int}, we obtain the required bound.
\end{proof}

%
%
%
%


\section*{Conclusions}
In this paper, we studied the asymptotic behavior of a controlled non-Markovian dynamical system under the combined effects of fast periodic sampling and small white noise. Our key contribution is to provide uniform-in-time control of its limiting behavior and of its fluctuations. The limiting stochastic dynamical system describing the fluctuations captures both the sampling and noise effects.  We approximate uniformly-in-time the pre-limit non-Markovian process by a simpler limiting Markovian process with time-independent bounds on the remainder.\\

\textbf{Acknowledgments.} This work was partially supported by the National Science Foundation {\sc dms} 2107856 and {\sc dms} 2311500. The authors would like to thank the referees for their positive feedback and helpful comments on the earlier version of the manuscript.

\bibliographystyle{alpha}
\bibliography{References}

\newcommand{\etalchar}[1]{$^{#1}$}
\begin{thebibliography}{CDG{\etalchar{+}}22}

\bibitem[BZ24]{budhiraja2024large}
Amarjit Budhiraja and Pavlos Zoubouloglou.
\newblock Large deviations for small noise diffusions over long time.
\newblock {\em Transactions of the American Mathematical Society, Series B},
  11(01):1--63, 2024.

\bibitem[CDG{\etalchar{+}}22]{crisan2022poisson}
Dan Crisan, Paul Dobson, Ben Goddard, Michela Ottobre, and Iain Souttar.
\newblock {P}oisson equations with locally-{L}ipschitz coefficients and
  uniform-in-time averaging for stochastic differential equations via strong
  exponential stability.
\newblock {\em arXiv:2204.02679}, 2022.

\bibitem[CDO21]{crisan2021uniform}
Dan Crisan, Paul Dobson, and Michela Ottobre.
\newblock Uniform-in-time estimates for the weak error of the {E}uler method
  for {SDE}s and a pathwise approach to derivative estimates for diffusion
  semigroups.
\newblock {\em Transactions of the American Mathematical Society},
  374(5):3289--3330, 2021.

\bibitem[Cer09]{cerrai2009khasminskii}
Sandra Cerrai.
\newblock A {K}hasminskii type averaging principle for stochastic
  reaction--diffusion equations.
\newblock {\em The Annals of Applied Probability}, 19(3):899--948, 2009.

\bibitem[CF09]{cerrai2009averaging}
Sandra Cerrai and Mark~I. Freidlin.
\newblock Averaging principle for a class of stochastic reaction--diffusion
  equations.
\newblock {\em Probability Theory and Related Fields}, 144(1-2):137--177, 2009.

\bibitem[DP21]{dhama2021asymptotic}
Shivam Dhama and Chetan~D. Pahlajani.
\newblock Asymptotic analysis of discrete-time models for linear control
  systems with fast random sampling.
\newblock In {\em 2021 Seventh Indian Control Conference (ICC)}, pages
  359--364. IEEE, 2021.

\bibitem[DP23]{dhama2023fluctuation}
Shivam Dhama and Chetan~D. Pahlajani.
\newblock Fluctuation analysis for a class of nonlinear systems with fast
  periodic sampling and small state-dependent white noise.
\newblock {\em Journal of Differential Equations}, 362:438--483, 2023.

\bibitem[FG20]{fang2020adaptive}
Wei Fang and Michael~B. Giles.
\newblock Adaptive {E}uler--{M}aruyama method for {SDE}s with nonglobally
  {L}ipschitz drift.
\newblock {\em The Annals of Applied Probability}, 30(2):526--560, 2020.

\bibitem[FS99]{FreidlinSowers}
Mark~I. Freidlin and Richard~B. Sowers.
\newblock A comparison of homogenization and large deviations, with
  applications to wavefront propagation.
\newblock {\em Stochastic Processes and their Applications}, 82:23--52, 1999.

\bibitem[FW12]{FW_RPDS}
Mark~I. Freidlin and Alexander~D. Wentzell.
\newblock {\em Random Perturbations of Dynamical Systems}.
\newblock Springer, third edition, 2012.

\bibitem[GST12]{GST-HybDynSys-book}
Rafal Goebel, Ricardo~G. Sanfelice, and Andrew~R. Teel.
\newblock {\em Hybrid dynamical systems: Modeling, stability and robustness}.
\newblock Princeton University Press, 2012.

\bibitem[Kha66]{has1966stochastic}
R.Z. Khasminskii.
\newblock On stochastic processes defined by differential equations with a
  small parameter.
\newblock {\em Theory of Probability \& Its Applications}, 11(2):211--228,
  1966.

\bibitem[Kha68]{khasminskij1968principle}
R.Z. Khasminskii.
\newblock On the principle of averaging the {I}t{\^o}'s stochastic differential
  equations.
\newblock {\em Kybernetika}, 4(3):260--279, 1968.

\bibitem[KS91]{KS91}
Ioannis Karatzas and Steven Shreve.
\newblock {\em Brownian Motion and Stochastic Calculus}, volume 113 of {\em
  Graduate Texts in Mathematics}.
\newblock Springer-Verlag New York, second edition, 1991.

\bibitem[NTC09]{NesicTeelCarnevale-TAC2009}
Dragan Nesic, Andrew~R. Teel, and Daniele Carnevale.
\newblock Explicit computation of the sampling period in emulation of
  controllers for nonlinear sampled-data systems.
\newblock {\em IEEE Transactions on Automatic Control}, 54(3):619--624, March
  2009.

\bibitem[PV01]{pardoux2001poisson}
E.~Pardoux and A.~Yu Veretennikov.
\newblock On the {P}oisson equation and diffusion approximation {I}.
\newblock {\em The Annals of Probability}, 29(3):1061--1085, 2001.

\bibitem[PV03]{pardoux2003poisson}
E.~Pardoux and A.~Yu Veretennikov.
\newblock On {P}oisson equation and diffusion approximation {II}.
\newblock {\em The Annals of Probability}, 31(3):1166--1192, 2003.

\bibitem[PV05]{pardoux2005poisson}
E.~Pardoux and A.~Yu Veretennikov.
\newblock On the {P}oisson equation and diffusion approximation {III}.
\newblock {\em The Annals of Probability}, 33(3):1111--1133, 2005.

\bibitem[RX21a]{rockner2021averaging}
Michael R{\"o}ckner and Longjie Xie.
\newblock Averaging principle and normal deviations for multiscale stochastic
  systems.
\newblock {\em Communications in Mathematical Physics}, 383(3):1889--1937,
  2021.

\bibitem[RX21b]{rockner2021diffusion}
Michael R{\"o}ckner and Longjie Xie.
\newblock Diffusion approximation for fully coupled stochastic differential
  equations.
\newblock {\em The Annals of Probability}, 49(3):1205--1236, 2021.

\bibitem[Spi13]{Spil-AppMathOptim}
Konstantinos Spiliopoulos.
\newblock Large deviations and importance sampling for systems of slow-fast
  motion.
\newblock {\em Applied Mathematics and Optimization}, 67:123--161, 2013.

\bibitem[Spi14]{spiliopoulos2014fluctuation}
Konstantinos Spiliopoulos.
\newblock Fluctuation analysis and short time asymptotics for multiple scales
  diffusion processes.
\newblock {\em Stochastics and Dynamics}, 14(03):1350026, 2014.

\bibitem[Ver00]{veretennikov2000large}
A.~Yu Veretennikov.
\newblock On large deviations for {SDE}s with small diffusion and averaging.
\newblock {\em Stochastic Processes and their Applications}, 89(1):69--79,
  2000.

\bibitem[YG14]{YuzGoodwin-book}
Juan~I. Yuz and Graham~C. Goodwin.
\newblock {\em Sampled-data models for linear and nonlinear systems}.
\newblock Springer, 2014.

\end{thebibliography}
\end{document}